\documentclass[12pt]{article}
\usepackage{amsmath, amssymb,paralist,amsthm}

\def\PC{\mathcal{P}}
\def\MC{\mathcal{M}}

\def\LC{\mathcal{L}}
\def\KC{\mathcal{K}}

\def\XC{\mathcal{X}}

\def\FC{\mathcal{F}}

\def\E{\mathbf{E}}

\def\P{\mathbf{P}}
\def\R{\mathbf{R}}

\def\x{\mathbf{x}}
\def\y{\mathbf{y}}
\def\1{\mathbf{1}}

\def\sym{\rm{sym}}

\def\al{\alpha}
\def\be{\beta}
\def\pa{\partial}
\def\ep{\epsilon}
\def\de{\delta}
\def\ga{\gamma}

\newcommand{\la}{\lambda}

\newcommand{\om}{\omega}
\newcommand{\Ga}{\Gamma}

\newcommand{\Om}{\Omega}
\newcommand{\La}{\Lambda}

\def\B{\mathbf{B}}
\def\D{\mathbf{D}}
\def\E{\mathbb{E}}
\def\R{\mathbf{R}}

\def\M{\mathcal{M}}

\def\P{\mathcal{P}}
\def\U{\mathcal{U}}

\newtheorem{prop}{Proposition}[section]
\newtheorem{theorem}{Theorem}[section]

\newtheorem{definition}{Definition}[section]
\newtheorem{example}{Example}[section]
\newtheorem{remark}{Remark}[section]
\newtheorem{assumption}{Assumption}[section]

\numberwithin{equation}{section}

\begin{document}

\centerline{\Large \bf Mean Field Games and Nonlinear Markov Processes\footnote{Supported by the AFOSR grant FA9550-09-1-0664 'Nonlinear Markov control processes and games'}\footnote{arXiv:1112.3744v}}
\bigskip
\bigskip
\centerline{\bf Vassili N.Kolokoltsov, Jiajie Li and Wei Yang}
\smallskip
\centerline{Department of Statistics, University of Warwick}
\centerline{Coventry, CV4 7AL, UK}
\centerline{\it v.kolokoltsov@warwick.ac.uk\quad jiajie.li@warwick.ac.uk\quad wei.yang@warwick.ac.uk}
\bigskip

\begin{abstract}
In this paper, we investigate the mean field games with $K$ classes of agents who are weakly coupled via the empirical measure. The underlying dynamics of the representative agents is assumed to be a controlled nonlinear Markov process associated with rather general integro-differential generators of L\'evy-Khintchine type (with variable coefficients),
with the major stress on applications to stable and stable-like processes, as well as their various modifications like tempered stable-like processes or their mixtures with diffusions. We show that nonlinear measure-valued kinetic equations describing the dynamic law of large numbers limit for system with large number $N$ of agents are solvable and that their solutions represent $1/N$-Nash equilibria for approximating systems of $N$ agents.
\end{abstract}

{\medskip\par\noindent
}


\section{Introduction}

\subsection{Main objectives}

The mean-field game (MFG) methodology represents one of the gems in the recent progress of stochastic control. It was developed independently by J.-M. Lasry and P.-L. Lions in France (where the term 'mean-field game' was coined, see \cite{LL2006}, \cite{LL2006a}, \cite{LL2007},\cite{GLL2010}) and by M. Huang, R.P. Malham\'e and P. Caines in Canada, where it was initially called the 'Nash certainty equivalence principle', see \cite {HMC03}, \cite{HMC05}, \cite{HCM3}, \cite{HCM07}, \cite{HCM10}, \cite{Hu10}. Mean-field game methodology aims at describing control processes with large number $N$ of participants by studying the limit $N\to \infty$ when the contribution of each member becomes negligible and their interaction is performed via certain mean-field characteristics, which can be expressed in terms of empirical measures. A characteristic feature of the MFG analysis is the study of a coupled system of a backward equation on functions (Hamilton-Jacobi-Bellman equation) and a forward equation on probability laws (Kolmogorov equation). The work on mean-field games so far was performed mostly for the underlying Markov process (describing individual evolutions) being a simple diffusion, where the method of McKean-Vlasov SDEs (describing the so called nonlinear diffusions) was available.

Meanwhile one of the authors of the present paper suggested in \cite{Ko07} the program of studying nonlinear Markov processes that describe the dynamic laws of large numbers for general Markov systems of interacting particles extending and unifying in a natural way various models of natural science including McKean-Vlasov diffusions, Smoluchovski and Boltzman evolutions, replicator dynamics of evolutionary games and many others. This program was then developed in some detail in monograph \cite{Ko10}. In this paper we aim to show that the machinery of nonlinear Markov processes can serve as a natural tool for studying mean-field games with the general underlying Markov dynamics of agents (not only diffusions). More specifically, the main consistency equation of MFG can be looked at as a coupling of a nonlinear Markov process with certain controlled dynamics. Using this link we develop the MFG methodology for a wide class of underlying Markov dynamics including in particular stable and stable-like processes, as well as various their modifications like tempered stable-like process or their mixtures with diffusions.

Moreover, our abstract approach yields essential improvements even for underlying processes being diffusions. In particular, as compared with \cite{HCM3}, it includes the case of diffusions coefficients (not only drifts) depending on empirical measures, it allows us to get rid of the assumption of small coupling (or composite gain), to prove (rather than just assume) the crucial sensitivity estimates (feedback regularity condition (37) in \cite{HCM3}), and finally to get a full prove of convergence rate of order $1/N$.


\subsection{The basic setting of MFG methodology and the strategy for its implementation}
\label{sebsecMFG methodol}

Let us explain now the main ideas, objectives and strategy of our analysis. Suppose a position of an agent is described by a point in a locally compact separable metric space $\XC$. A position of $N$ agents is then given by a point in the power $\XC^N=\XC\times \cdots \times \XC$ ($N$ times). Hence the natural state space for describing the variable (but not vanishing) number of players is the union $\hat \XC=\cup_{j=1}^{\infty} \XC^j$. We denote by
$C_{\sym}(\XC^N)$ the Banach spaces of symmetric (with respect to permutation of all arguments) bounded continuous functions on $\XC^N$ and by $C_{\sym}(\hat \XC)$ the corresponding space of functions on the full space $\hat \XC$. We denote the elements of $\hat \XC$ by bold letters, say $\x$, $\y$.

Reducing the set of observables to $C_{\sym}(\hat \XC)$ means effectively that our state space is not $\hat \XC$ (or $\XC^N$ in case of a fixed number of particles) but rather the quotient space $S\hat \XC$ (or $S\XC^N$ resp.) obtained with respect to the action of the group of permutations, which allows the identifications $C_{\sym}(\hat \XC)=C(S\hat \XC)$ and $C_{\sym}(\XC^N)=C(S\XC^N)$. Clearly $S\hat \XC$ can be identified with the set of all finite collections of points from $\XC$, the
order being irrelevant.

A key role in the theory of measure-valued limits of interacting particle systems is played by the inclusion $S\hat \XC$ to $\PC(\XC)$ (the set of probability laws on $\XC$) given by
\begin{equation}
\label{eqcorrsetparticlesandmes}
 \x=(x_1,...,x_N) \mapsto
\frac1N (\delta_{x_1}+\cdots +\delta_{x_N})=\frac1N \delta_{\x},
 \end{equation}
which defines a bijection between $S\XC^N$ and the subset $\PC^N_{\delta}(\XC)$ (of normalized sums of Dirac's masses) of $\PC(\XC)$.
This bijection extends to the bijection of $S\hat \XC$ to
 \[
 \PC_{\delta}(\XC):=\cup_{N=1}^{\infty}\PC^N_{\delta}(\XC)\subset \PC(\XC),
 \]
that can be used to equip $S\hat \XC$ with the structure of a metric space by pulling back any distance on $\PC(\XC)$ that is compatible with its weak topology.

\begin{remark}
With a slight abuse of notation, we use $\delta_{\x}$ to denote the sum of the Dirac's measures on $\XC$, i.e.
$$\delta_{\x}:=\delta_{x_1}+\cdots +\delta_{x_N}.$$
So $\delta_{\x}$ is not the Dirac measure on $\hat \XC$ with support at point $\x$.
\end{remark}

Let $\{A[t,\mu,u]\}$ be a family of generators of Feller processes in $\XC$, where $t\ge 0$, $\mu \in \PC(\XC)$  and $u\in \U$ (a metric space interpreted as a set of admissible controls). Assume also that a mapping $\gamma: \R^+ \times \XC \to \U$ is given. For any $N$, let us define the following (time-dependent) family of operators (pre-generators) on $C_{\sym}(\XC^N)$ describing $N$ mean-field interacting agents:

\begin{equation}
\widehat A^N_t[\gamma]f(\x)= \widehat A^N_t[\gamma]f(x_1,\cdots , x_N):=\sum_{i=1}^N A^i[t,\mu,u_i]f (x_1,\cdots , x_N),
\end{equation}
where
\[
\mu= \frac1N \sum_{i=1}^N \de_{x_i}=\frac1N \de_{\x}
\]
is the empirical distribution of agents, $u^i=\gamma (t,x_i)$ and $A^i[t,\mu,u_i]f$ means the action of the
operator $A[t,\mu,u_i]$ on the $i$th variable of the function $f$.
Let us assume that the family $\widehat A^N_t[\gamma]$ generates a Markov process $X^N=\{X^N(t)=(X^N_1(t),\dots,X^N_N(t):t\geq 0)\}$ on $\XC^N$ for any $N$. We shall refer to it as a {\it controlled  (via control $\gamma$) process of $N$ mean-field interacting agents}.

\begin{remark} In the terminology of statistical mechanics the operator $\hat A_t[\gamma]$ (considered for all $N$, i.e. lifted naturally to the whole space $C_{\sym}(\hat \XC)$) should be called the second quantization of
$A[t,\mu,u]$.
\end{remark}

Using mapping \eqref{eqcorrsetparticlesandmes},
we can transfer our process of $N$ mean-field interacting agents from $S\XC^N$ to $\PC^N_{\delta}(\XC)$.
This leads to the following operator on $C(\PC^N_{\delta}(\XC))$:
 \begin{equation}\label{eq0Nparticlegen}
\widehat A^N_t[\gamma]F(\de_{\x}/N)= \widehat A^N_t[\gamma] f(\x)
=\sum_{i=1}^N A^i[t,\mu,u_i]f (x_1,\cdots , x_N),
\end{equation}
where $f(\x)=F(\de_{\x}/N)$ and $\x=(x_1,\cdots, x_N)$.
Let us calculate the action of this operator on linear functionals $F$, that is on the functionals of the form
\begin{equation}\label{linearF}
F^g(\mu)=(g,\mu)=\int g(x) \mu (dx)
\end{equation}
for a $g\in C(\XC)$. Denoting $g^{\oplus} (\x)=\sum_{i=1}^N g(x_i)$ for $\x=(x_1,\cdots x_N)$ we get
\begin{equation}
\begin{split}
\label{eqNparticlegen}
&\widehat A^N_t[\gamma]F^g(\de_{\x}/N)= \frac1N \left( \widehat A^N_t[\gamma] g^{\oplus}\right)(x_1,\cdots , x_N)\\
&=\frac1N \sum_{i=1}^N \left(A[t, \de_{\x}/N, \gamma (t,x_i)]g\right)(x_i)
=\left(A[t,\de_{\x}/N, \gamma (t,.)]g,\de_{\x}/N \right).
\end{split}
\end{equation}
Hence, if $\mu^N_t= \de_{\x}/N\to \mu_t\in\PC(\XC)$ as $N\to \infty$, we have
$$\widehat A^N_t[\gamma]F^g(\de_{\x}/N) \to \left(A[t,\mu^N_t, \gamma (t,.)]g,\mu^N_t \right), \quad \text{as} \quad N\to \infty,$$
so that the evolution equation
\begin{equation}
\dot F_t=\widehat A^N_t[\gamma]F_t
\end{equation}
of our controlled process of $N$ mean-field interacting agents, for the linear functionals of the form $F^g_t(\mu)=(g,\mu_t(\mu))$ turns to the equation
\begin{equation}
\label{eqkineqmeanfield}
 {d \over dt} (g, \mu_t)
 =(A[t,\mu_t, \gamma (t,.) ]g, \mu_t), \quad \mu_0=\mu.
\end{equation}

We call this equation the {\it general kinetic equation} in weak form. It should hold for $g$ from a suitable class of test functions.  This limiting procedure will be discussed in detail later on.

Let us explain how the mapping $\gamma$ pops in from individual controls. Assume that the objective of each agent is to maximize (over a suitable class of controls $\{u.\}$) the payoff
  \[
   \E \left[ \int_t^T J(s,X^N_i(s),\mu^N_s,u_s) \, ds +V^T (X^N_i(T))\right],
\]
consisting of running and final components, where the functions $J: \R^+ \times \XC \times \PC(\XC)\times \U \to \R$ and $V^T: \XC\to \R$, and the final time $T$ are given, and where $\{\mu.\}$ is the family of the empirical measures of the whole process
  \[
  \mu^N_s=\frac{1}{N} (\de_{X^N_1(s)}+\cdots +\de_{X^N_N(s)}), \quad t\leq s\leq T.
  \]
 By dynamic programming (and assuming appropriate regularity), if the dynamics of empirical measures $\mu_s$ is given, the optimal payoff
  \[
 V_N(t,x)=\sup_{u.} \E \left[ \int_t^T J(s,X(s),\mu^N_s,u_s) \, ds +V^T (X(T))\right]
\]
of an agent starting at $x$ at time $t$ should satisfy the HJB equation
\begin{equation}
\label{HJB-N}
\frac{\pa V_N(t,x)}{\pa t}+\max_u \big(J(t,x,\mu^N_t,u) +A[t,\mu^N_t,u] V_N (t,x)\big)=0
\end{equation}
with the terminal condition $V_N(T,.)=V^T(\cdot)$. If $\mu^N_t\to \mu_t\in\PC(\XC)$ as $N\to \infty$, then it is reasonable to expect that the solution of \eqref{HJB-N} converges to the solution of the equation
\begin{equation}
\label{HJB}
\frac{\pa V(t,x)}{\pa t}+\max_u \big(J(t,x,\mu_t,u) +A[t,\mu_t,u] V (t,x)\big)=0.
\end{equation}

\begin{remark}
To shorten the formulas used below for the analysis of HJB equation \eqref{HJB-N}, we assume that the final cost function $V^T$ depends only on the terminal position $X(T)$, but not on $\mu_T$. Only minor modifications are required, if $V^T$ does depend on $\mu_T$.
\end{remark}

Assume HJB equation \eqref{HJB} is well posed and the $\max$ is achieved at one point only. Let us denote this point of maximum by $u=\Gamma (t,x, \{\mu_{\ge t}\})$.

Thus, if each agent chooses the control via HJB \eqref{HJB}, given an empirical measure $\hat \mu$, i.e. with
\begin{equation}
\label{couplHJBkin}
\ga (t,x)=\Ga (t, x, \{\hat \mu_{\ge t}\}),
\end{equation}
this $\ga$ specifies a nonlinear Markov evolution $\{\mu_t\}_{t\geq 0}$ via kinetic equation \eqref{eqkineqmeanfield}. The corresponding {\it MFG consistency (or fixed point) condition} $\{\hat \mu_.\}=\{\mu_.\}$
leads to the equation
\begin{equation}
\label{eqkineqmeanfieldnonM}
 {d \over dt} (g, \mu_t)
 =(A[t,\mu_t, \Gamma (t,.,\{\mu_{\ge t}\}) ]g, \mu_t),
\end{equation}
which expresses the {\it coupling} of the nonlinear Markov process specified by \eqref{eqkineqmeanfieldnonM}
and the optimal control problem specified by HJB \eqref{HJB}.
It is now reasonable to expect that if the number of agents $N$ tends to infinity in such a way that the limiting evolution is well defined and satisfies the limiting equation \eqref{eqkineqmeanfieldnonM}
with $\Gamma$ chosen via the solution of the above HJB equation, then the control $\gamma$ and the corresponding payoffs represent the $\ep$-Nash equilibrium for the controlled system of $N$ agents, with $\ep \to 0$, as $N\to \infty$. This statement (or conjecture) represents the essence of the {\it MFG methodology}.

Under certain assumptions on the family $A[t,\mu,u]$, we are going to justify this claim by carrying out the following tasks:

 T1) Proving the existence of solutions to the Cauchy problem for coupled kinetic equations \eqref{eqkineqmeanfieldnonM} within an appropriate class of feedback $\Gamma$ and the well-posedness for the uncoupled equations
 \eqref{eqkineqmeanfield}. Notice that we are not claiming uniqueness for \eqref{eqkineqmeanfieldnonM}. It is difficult to expect this, as in general Nash equilibria are not unique. At the same time, it seems to be an important open problem to better understand this non-uniqueness by describing and characterizing specific classes of solutions. On the other hand, well-posedness for the uncoupled equations \eqref{eqkineqmeanfield} is crucial for further analysis.

 T2) Proving the well-posedness of the Cauchy problem for the (backward) HJB equation \eqref{HJB}, for an arbitrary flow $\{\mu.\}$ in some class of regularity, yielding the feedback function $\Gamma$ in the class required by T1).
 This should include some sensitivity analysis of $\Gamma$ with respect to the functional parameter $\{\mu_.\}$, which will be needed
 to show that approximating the limiting MFG distribution $\{\mu.\}$ by approximate $N$-particle empirical measures yields also
an approximate optimal control. To perform this task, we shall assume here additionally that the operators
$A[t,\mu,u]$ in \eqref{HJB} can be decomposed into the sum of a controlled 1st order term and a term that does not depend on control and generates a propagator with certain smoothing properties. This simplifying assumption
allows to work out the theory with classical (or at least mild) solutions of HJB equations. Without this assumption, one would have to face additional technical complications related to viscosity solutions.

 T3) Showing the convergence of the $N$-particle approximations, given by generators \eqref{eqNparticlegen} to the limiting evolution \eqref{eqkineqmeanfield}, i.e. the {\it dynamic laws of large numbers} (LLN), for a class of controls
$\ga$ arising from \eqref{couplHJBkin} with a fixed $\{\hat \mu.\}$, where $\Ga$ is from the class required for the validity of T1) and T2). Here one can use either more probabilistic compactness and tightness (on Skorokhod paths spaces) approach, or a more analytic method (suggested in \cite{Ko07}) via semigroups of linear operators on continuous functionals of measures. In this paper, we shall use the second method, as it yields more precise convergence rates. For the analysis of the convergence of the corresponding semigroups the crucial ingredient is the analysis
 of smoothness ({\it sensitivity}) of the solutions to kinetic equations
\eqref{eqkineqmeanfield} with respect to initial data. The rates of convergence in LLN imply directly the corresponding rather precise estimates for the so-called {\it propagation of chaos} property of interacting particles.

 T4) Finally, combining T2) and T3), one has to show that thus obtained strategic profile \eqref{couplHJBkin}
 with $\{\hat \mu.\}=\{\mu.\}$ represents an $\ep$-{\it equilibrium} for $N$ agents system with $\ep\to 0$, as $N\to \infty$. Actually we going to prove this with $\ep=1/N$ using the method of {\it tagged particles} in our control setting.

\begin{remark}
A similar, but slightly different statement of the MFG methodology would be the claim that a sequence of Nash equilibria for $N$ particle game converges to a profile $\Gamma$ solving \eqref{eqkineqmeanfieldnonM}. By our methods it would not be difficult to justify this claim if existence and uniqueness of a Nash equilibrium for each $N$-player stochastic differential game were given. Our formulation of MFG methodology allows us to avoid discussion of this nontrivial question.
\end{remark}


\subsection{Discrete classes of agents}
\label{subescdiscrete}

Let us specify our model a bit further.

Of particular interest are the models with the one-particle space $\XC$ having a spatial and a discrete components, the latter interpreted as a type of an agent. Thus let $\XC=\R^d \times \KC$, where $\KC$ is either a finite or denumerable set. In this case, functions from $C(\XC)$ can be represented by sequences $f=(f_i)_{i\in \KC}$ with each $f_i \in C(\R^d)$, the probability laws on $\XC$ are similarly given by the sequences $\mu=(\mu_i)_{i\in \KC}$ of positive measures on $\R^d$ with the masses totting up to one.

The operators $A$ in $C(\XC)$ are specified by operator-valued matrices $\{A_{ij}\}$, $i,j \in \KC$, with $A_{ij}$ being an operator in $C(\R^d)$, so that $(Af)_i=\sum_{j\in \KC} A_{ij}f_j$. It is not difficult to show (see \cite{BeK3} for details) that for such a matrix $A$ to define a conditionally positive conservative operator in $C(\XC)$ (in particular, a generator of a Feller process) it is necessary that $A_{ij}$ for $i\neq j$ are integral operators
    \[
    (A_{ij} f)(z)=\int_{\R^d} (f_j(y)-f(z))\nu_{ij}(z,dy)
    \]
with a bounded (for each $z$) measure $\nu_{ij}(z,dy)$, and the diagonal terms are given by the {\it L\'evy-Khintchin} type operators ($i\in \KC$):
\[
 A_{ii}f(z)=\frac{1}{2}(G_i(z)\nabla,\nabla)f(z)+ (b_i(z),\nabla f(z))
\]
\begin{equation}
\label{fellergenerator}
 +\int_{\R^d}
(f(z+y)-f(z)-(\nabla f (z), y){\bf 1}_{B_1}(y))\nu_i (z,dy),
\end{equation}
with $G_i(z)$ being a symmetric non-negative matrix, $\nu_i(z,.)$ being a L\'evy
measure on $\R^d$, i.e.
\begin{equation}
\label{condlevy0}
\int_{\R^d} \min (1,|y|^2)\nu_i (z, dy) <\infty, \quad \nu (\{0\})=0,
\end{equation}
depending measurably on $z$, and where ${\bf 1}_{B_1}$ denotes, as usual, the indicator function of the unit ball in $\R^d$.

Operators $A_{ij}$ with $i\neq j$ describe the {\it mutation (migration)} between the types. Leaving the discussion of mutations of infinite number types to another publication, we shall concentrate here,
 for simplicity, on the case when $\KC$ is a finite set $\{1,\cdots, K\}$ and no mutations are allowed, that is all non-diagonal operators $A_{ij}$ with $i\neq j$ vanish. Hence $A$ will be given by a diagonal matrix with the diagonal terms $A_i=A_{ii}$ of type \eqref{fellergenerator}.

Let us assume additionally that each agent can control only its drift, that is the diagonal generators have the form
\begin{equation}
\label{fellergeneratorwithdriftcont}
A_i[t,\mu,u]f(z)=(h_i(t,z,\mu,u), \nabla f(z)) + L_i[t,\mu] f(z), \quad i=1,\cdots, K,
\end{equation}
with $L_i$ of form \eqref{fellergenerator}, i.e.
\begin{equation}\label{fellergenerator 2}
\begin{split}
&L_i[t,\mu]f(z)=\frac{1}{2}(G_i(t,z,\mu)\nabla,\nabla)f(z)+ (b_i(t,z,\mu),\nabla f(z))\\
 &+\int_{\R^d} (f(z+y)-f(z)-(\nabla f (z), y){\bf 1}_{B_1}(y))\nu_i (t,z,\mu,dy)
\end{split}
\end{equation}
with the coefficients $G_i, b_i, \nu_i$ depending on $t \in \R^+$ and $\mu=(\mu_1,\cdots, \mu_K) \in \PC (\XC)$ as parameters.

\begin{remark}
If, for a given (probability) measure flow $\{\mu_t\}_{t\in[0,T]}$, the operators
$L[t,\mu_t]=(L_1, \cdots, L_K)[t, \mu_t]$ generate a Markov process $\{R_t[\mu_t]\}_{t\in[0,T]}=\{(R_t^1[\mu_t], \cdots,  R_t^K[\mu_t] )\}_{t\in[0,T]}$, one can write a stochastic differential equation (SDE) corresponding to the generator given in (\ref{fellergeneratorwithdriftcont}) as
\[
dX_t^i=h_i(t,X_t^i, \mu_t, u^i_t)\, dt+dR_t^i[\mu_t],\quad i=1,\cdots, K.
\]
If $\mu_t$ are required to coincide with the laws of $X_t^i$,  for all $t\in[0,T]$, these equations take the form of
 {\it SDEs driven by nonlinear L\'evy noises}, developed in \cite{Ko11a} and Chapter 3 of \cite{Ko10}.

The initial work on the mean field games, done by Lions et al. and Caines et al., dealt with
the processes $R_t[\mu]$ being Brownian Motions without dependence on $\mu$. In our framework, this underlying process is extended to an arbitrary Markov process with a generator (\ref{fellergenerator 2}) depending on $\mu$.
\end{remark}

In the main kinetic equation \eqref{eqkineqmeanfield}, we shall then have
$(g,\mu_t)=\sum_{i=1}^K (g_i,\mu_{i,t})$ and
\[
A[t,\mu_t, \gamma (t,.) ]g=\{A_i[t,\mu_t, \gamma (t,.) ]g_i\}_{i=1}^K
\]
with
\begin{equation} \label{Ai}
 A_i[t,\mu, \gamma (t,.) ]g_i(z)
 =(h_i(t,z,\mu,\gamma (t,.)), \nabla g_i(z)) + L_i[t,\mu] g_i(z).
\end{equation}

HJB equation (\ref{HJB}) now decomposes into a collection of HJB equations for each class of agents, written as
\begin{equation}
\label{HJB 2}
\frac{\partial V^i (t,x)}{\partial t}+ H_t^i(x,\nabla V^i (x), \mu_t)+ L_i[t,\mu_t]V^i(t,x)=0
\end{equation}
where
\begin{equation}
\label{HJB 2a}
H_t^i(x,p, \mu_t):=\max_{u\in\U} \{  h_i(t,x,\mu_t,u)p+J_i(t,x,\mu_t,u)\}.
\end{equation}
We have assumed the resulting feedback control is unique (i.e. argmax in \eqref{HJB 2a} is unique). Let us give two (related) basic examples of such a situation.

\begin{example}[$H_\infty$-optimal control, see \cite{McEn} for its systematic presentation]
For each $i$, the running cost function $J_i$ is quadratic in $u$, i.e.
\[
J_i(t,x,\mu,u)=\alpha_i (t,x,\mu)-\theta_i (t,x,\mu)u^2
\]
and the drift coefficient $h_i$ is linear in $u$, i.e.
\[
h_i(t,x,\mu,u)=\beta_i (t,x,\mu)u,
\]
where the functions $\alpha_i, \beta_i, \theta_i: [0,T]\times \R^d \times \PC(\R^d)\to \R$
 and $\theta_i (t,x,\mu)> 0$ for any $(t,x,\mu)$. Thus we are maximising a quadratic function over control $u$. It is easy to get an explicit formula of the unique point of maximum, i.e.
\[
u=\frac{\beta }{2\theta}(t,x,\mu)p.
\]
Thus HJB equation (\ref{HJB 2}) rewrites as
\[
\frac{\partial V^i (t,x)}{\partial t}+ \frac{\beta_i^2}{4\theta_i}(t,x,\mu)(\nabla V^i)^2(t,x)
 +\al_i(t,x,\mu) + L_i[t,\mu_t]V^i(t,x)=0
\]
which is a generalized backward Burger's equation.
\end{example}

\begin{example} Assume, for each $i$,  $h_i(t,x,\mu,u)=u$ and $J_i(t,x,\mu,u)$ is a strictly concave smooth function of $u$.
Then $H^i_t$ is the Legendre transform of $-J$ as a function of $u$, and the unique point of maximum in
\eqref{HJB 2a} is therefore $u=\pa H^i_t/\pa p$. Then kinetic equation
\eqref{eqkineqmeanfieldnonM} takes the form
\begin{equation}
\label{eqLionsset}
\frac{d}{dt}  (g, \mu_t)=\sum_{k=1}^K (L_i (t, \mu_t) g_i
+ \frac{\pa H^i_t}{\pa p} (x,p,\mu_t)|_{p=\nabla V^i(x)} \nabla g_i, \mu_{i,t}).
\end{equation}
If $K=1$ and $J_i(t,x,\mu,u)$ has the decomposition
$$J_i(t,x,\mu,u)=\tilde{V}(x,\mu)+\tilde J(x,u)$$
for $\tilde{V}:\R^d\times \P(\R^d)$ and $\tilde J: \R^d\times \U$, and $L_i (t,\mu)=\Delta$,  the corresponding
system of coupled equations \eqref{eqLionsset} and \eqref{HJB 2} turns to system (2) of
\cite{LL2007} (only there the kinetic equation is written in the strong form and in reverse time).
\end{example}

Concrete examples, for which our abstract results hold, are presented in Section \ref{examples}.


\subsection{Plan of the paper}

The rest of the paper is organized as follows. In Section \ref{secabstractnonMark}, we prove basic existence and uniqueness results for the extension of forward kinetic equation \eqref{eqkineqmeanfield} obtained by coupling $\gamma$ with $\{\mu_t\}_{t\in[0,T]}$. In Section \ref{secnonlinMark}, we analyze more specifically equation \eqref{eqkineqmeanfield}
 without any coupling and give a probabilistic interpretation of its solutions as nonlinear Markov processes, as well as some related regularity results. Then the basic examples are discussed in Section \ref{examples}.
 In Section \ref{Cauchy HJB}, we are concerned with the well-posedness of the backward HJB equation (\ref{HJB}), i.e. the existence and uniqueness of mild solutions to the equation with the given terminal boundary condition. Our main additional assumption will be a smoothing property of the semigroup of an underlying Markov process that defines an uncontrolled part of the evolution of agents (generated by the operators $L_i$ in \eqref{fellergeneratorwithdriftcont}). Sensitivity analysis on the solution of this HJB equation will be carried out in Section \ref{Sensitivity} completing in essence tasks T1) and T2) outlined above
 in Subsection \ref{sebsecMFG methodol}, for models with discrete classes from Subsection \ref{subescdiscrete}.
 Section \ref{secsmoothdepnonlinMark} deals
 with smooth dependence (sensitivity) of nonlinear Markov processes with respect to initial data. This is crucial
 for the proof of the dynamic laws of large numbers (LLN) established, in an appropriate form with rates of convergence, in Section \ref{Convergence} solving task T3) from Subsection \ref{sebsecMFG methodol}.
  Section \ref{secepequilibrium} completes our programme by justifying the approximation of $N$ players stochastic games by a mean-field limit.


\subsection{Main notations}

The following basic notations will be used:

For a Banach space $\B$, $\B^*$ is the dual Banach space with the norm $\|\mu\|_{\B^*}=\sup_{\|f\|_{\B}\leq 1}|(f,\mu)|$

For a linear operator $L$ between the Banach spaces $\D$ and $\B$, we use the standard notation for its norm
$$\|L\|_{\D\mapsto \B} = \sup_{\|g\|_\D=1}\|Lg\|_\B;$$
in particular, $\|L\|_\B = \|L\|_{\B\mapsto \B}$

$C([0,T], \B)$ is the Banach space of continuous functions $t \rightarrow \mu_t \in \B$  with the norm $\sup_{t\in[0,T]}\|\mu_t\|_{\B}$

If $\M$ is a closed subset of $\B$, then

$C([0,T], \M(\B))$ is a closed subset of $C([0,T], \B)$, consisting of continuous functions $t \rightarrow \mu_t \in \M(\B)$; using $\M(\B)$ rather than simply $\M$, stresses the topology used

$C_{Lip}([0,T], \M(\B))$ is a subset of $C([0,T], \M(\B))$, consisting of Lipschitz continuous functions $t \rightarrow \mu_t \in \M(\B)$

$C_{\mu}([0,T], \B)$ (resp. $C_{\mu}([0,T], \M(\B))$), a closed subset of $C([0,T], \B)$ (resp. $C([0,T], \M(\B))$), consists of continuous curves $\xi_s\in \B$ (resp. $\xi_s\in\M(\B)$), $s\in [0,T]$ with a fixed initial data $\xi_0=\mu \in \MC$

$C^T_{\mu}([0,T], \B)$ (resp. $C^T_{\mu}([0,T], \M(\B))$), a closed subset of $C([0,T], \B)$ (resp. $C([0,T], \M(\B))$), consists of continuous curves $\xi_s\in \B$ (resp. $\xi_s\in\M(\B)$), $s\in [0,T]$ with a fixed terminal data $\xi_T=\mu$

For a locally compact metric space $\XC$,

$C_{\infty}(\XC)$ is the Banach space of bounded continuous functions $f$ on $\XC$ with $\lim_{x\rightarrow\infty}f(x)=0$, equipped with sup-norm

$\mathbf{M}^{sign}(\XC)$, the dual space of $C_{\infty}(\XC)$, is the Banach space of signed Borel measures on $\XC$;
$(f,\mu)=\int f(x)\mu(dx)$ denotes the usual pairing

$\mathbf{M}^{+}(\XC)$, a subset of $\mathbf{M}^{sign}(\XC)$, is the set of positive Borel measures on $\XC$

$\P(\XC)$ is the set of probability measures on $\XC$

$C_{Lip}(\mathbf{R}^d)$ is the Banach space of bounded Lipschitz continuous functions $f$ on $\mathbf{R}^d$ with the norm $\|f\|_{{Lip}}=\sup_x|f(x)|+\sup_{x\neq y}\frac{|f(x)-f(y)|}{|x-y|}$

$C^n(\mathbf{R}^d)$ is the Banach space of $n$ times continuously differentiable and bounded functions $f$ on $\mathbf{R}^d$  such that each derivative up to and including order $n$ is bounded, equipped with norm $\|f\|_{C^n}$ which is the supremum of the sums of all the derivatives up to and including order $n$

$C_{\infty}^n(\mathbf{R}^d)$ is a closed subspace of $C^n(\mathbf{R}^d)$ with $f$ and all its derivatives up to and including order $n$ belonging to $C_{\infty}(\mathbf{R}^d)$; sometimes we use short notations, $C_{\infty}^n$ and $C_{\infty}$, respectively

\section{Abstract anticipating kinetic equations}
\label{secabstractnonMark}

This section extends some existence and uniqueness results for nonlinear Markov processes from \cite{Ko10} to a rather general anticipating (non-Markovian) setting.

Let $\D$ be a dense subset of $\B$, which is itself a Banach space with the norm $\|\,\|_\D \geq \|\,\|_\B$.

  A deterministic dynamic in the dual Banach space $\B^*$ can be naturally specified by vector-valued ordinary differential equation
\begin{equation}
\label{ODE1}
	\dot{\mu_t}=\Omega(t, \mu_t)
\end{equation}
with given initial value $\mu_0\in \B^*$, where $\Omega$ is a nonlinear (possibly unbounded) operator in $\B^*$. One can write equation (\ref{ODE1}) in the weak form as the equation
\begin{equation}
\label{ODE2}
(f,\dot{\mu_t})=(f,\Omega(t,\mu_t))
\end{equation}
that must hold for some class of test functions $f\in \B$. In many applications, equation (\ref{ODE2}) can be represented in the form
\begin{equation}
\label{ODE3}
\frac{d}{dt}(f,\mu_t)=(A[t,\mu_t]f,\mu_t)
\end{equation}
where the mapping $(t,\eta)\mapsto A[t,\eta]$ is from $\R^+\times \B^*$ to linear operators $A[t, \eta]: \D\mapsto \B$ and for each pair $(t,\eta)$, $A[t,\eta]$ generates a semigroup in $\B$. We call equation (\ref{ODE3}) the {\it general Markovian kinetic equation}, as it contains most of the basic equations of non-equilibrium statistical mechanics and evolutionary biology. For an extensive discussion of its properties and applications we refer to monograph
\cite{Ko10}.
Of major interest is the case when $\B^*$ is the space of measures on a locally compact space $\XC$,
 i.e. $\B^*=\mathbf{M}^{sign}(\XC)$ and $A[t, \eta]$ generate Feller processes, but in this section we shall work in an abstract setting.

 For the study of MFG, we need a anticipating extension of equation (\ref{ODE3}), that is an {\it anticipating} or {\it non-Markovian kinetic equation} of the form
\begin{equation}
\label{Non Markovian KE}
	\frac{d}{dt}(f,\mu_t)=(A[t,\{\mu_s\}_{0\leq s\leq T}]f,\mu_t)
\end{equation}
where the mapping $(t,\{\eta.\})\mapsto A[t,\{\eta.\}]$ is from $\R^+\times C_{\mu}([0,T], \B^*)$ to linear operators $A[t, \{\eta.\}]: \D\mapsto \B$.

Two particular cases are of major importance. When the operators $A$ only depend on the past, i.e. $\{\mu_s\}_{0\leq s\leq T}=\{\mu_{\leq t}\}$, equation (\ref{Non Markovian KE}) can be called a {\it non-anticipating kinetic equation}; when the generators $A$ only depend on the future, i.e. $\{\mu_s\}_{0\leq s\leq T}=\{\mu_{\geq t}\}$, equation (\ref{Non Markovian KE}) can be called a {\it forecasting kinetic equation}.

Non-anticipating equations can be seen as analytic analogs of SDE with adapted coefficients, and their well-posedness can be obtained by similar methods. But for MFG, we need much more complicated forecasting equations. Here we prove local well-posedness and global existence results for general equations \eqref{Non Markovian KE} (including forecasting ones) and global well-posedness for non-anticipating case.

Let us start by recalling the basic properties of propagators. For a set $S$, a family of mappings $U^{t,r}$ from $S$ to itself, parametrized by the pairs of numbers $r\leq t$ (resp. $t\leq r$) from a given finite or infinite interval is called a {\it (forward) propagator} (resp. a {\it backward propagator}) in $S$, if $U^{t,t}$ is the identity operator in $S$ for all $t$ and the following {\it chain rule}, or {\it propagator equation}, holds for $r\leq s\leq t$ (resp. for $t\leq s\leq r$):
$$U^{t,s}U^{s,r}=U^{t,r}.$$

A backward propagator ${U^{t,r}}$ of bounded linear operators on a Banach space $\B$ is called {\it strongly continuous} if the operators ${U^{t,r}}$ depend strongly continuously on t and r. By the principle of uniform boundedness if ${U^{t,r}}$ is a strongly continuous propagator of bounded linear operators, then the norms of ${U^{t,r}}$ are bounded uniformly for $t,r$ from any compact interval.

Suppose ${U^{t,r}}$ is a strongly continuous backward propagator of bounded linear operators on a Banach space with a common invariant domain $\D$. Let ${L_t}$, $t\geq0$, be a family of bounded linear operators $\D\mapsto \B$ depending strongly measurably on $t$. Let us say that the family ${L_t}$ {\it generates} ${U^{t,r}}$ on $\D$ if,
for any $f\in \D$, the equations
\begin{equation}
\label{Generates}
       \frac{d}{ds}U^{t,s}f = U^{t,s}L_sf, \quad \frac{d}{ds}U^{s,r}f = -L_sU^{s,r}f, \quad 0\leq t\leq s\leq r,
\end{equation}
hold a.s. in $s$ (with the derivatives taken in the topology of $B$), that is there exists a set of zero-measure $S$ in $\R$ such that for all $t\leq r$ and all $f\in \D$ equations (\ref{Generates}) hold for all $s$ outside $S$, where the derivatives exist in the Banach topology of $\B$. In particular, if the operators $L_t$ depend strongly continuously on $t$ (which will be always the case in this paper),  equations (\ref{Generates}) hold for all $s$ and $f\in \D$, where for $s=t$ (resp. $s=r$) it is assumed to be only a right (resp. left) derivative. In the case of propagators in the space of measures, the second equation in (\ref{Generates}) is called {\it  the backward Kolomogorov equation}.

One often needs to estimate the difference of two propagators when the difference of their generators is available.
To this end, we shall often use the following rather standard trick that we formulate now in abstract form.

\begin{prop}
\label{prop-propergatorProperty}
Let ${L_t^i}$, $i=1,2$, $t\geq 0$, be two families, continuous in $t$, of bounded linear operators $\D\mapsto \B$ ($\D$ and $\B$ are two Banach spaces equipped with a continuous inclusion $\D\to \B$) and ${U_i^{t,r}}$ be two backward propagators in $\B$
with $\|U_i^{t,r}\|_\B \le c_1$, $i=1,2$, s.t. for any $f\in \D$ equations \eqref{Generates} hold in $\B$ for both pairs $(L_i,U_i)$. Let $\D$ be invariant under $U_1^{t,r}$ and $\|U_1^{t,r}\|_\D \le c_2$. Then
\begin{equation}
\label{PropergatorProperty }
U_2^{t,r}-U_1^{t,r}=\int_t^rU_2^{t,s}(L^2_s-L_s^1)U_1^{s,r}ds
\end{equation}
and
\begin{equation}
\label{PropergatorPropertya}
\|U_2^{t,r}-U_1^{t,r}\|_{\D\to \B} \le c_1c_2(r-t) \sup_{t\leq s\le r}\|L^1_s-L_s^2 \|_{\D\to \B}.
\end{equation}
\end{prop}

\proof
Define an operator-valued function $Y(s):=U_2^{t,s}U_1^{s,r}$. Since $U_i^{t,t}$ are identity operators, $Y(r)=U_2^{t,r}$ and $Y(t)=U_1^{t,r}$. By (\ref{Generates}), we get
\begin{equation*}
\begin{split}
U_2^{t,r}-U_1^{t,r}&=U_2^{t,s}U_1^{s,r}\big|_{s=t}^r=\int_t^r\frac{d}{ds} \left(U_2^{t,s}U_1^{s,r} \right) ds\\
&=\int_t^r U_2^{t,s} L_s^2 U_1^{s,r}-U_2^{t,s} L_s^1 U_1^{s,r}ds\\
&=\int_t^rU_2^{t,s}(L^2_s-L_s^1)U_1^{s,r}ds,
\end{split}
\end{equation*}
which implies both \eqref{PropergatorProperty } and \eqref{PropergatorPropertya}.
\qed

As a consequence we get a simple result on the continuous dependence of propagators on parameters that can be considered as a starting point for later developed sensitivity.

\begin{prop}
\label{prop-propergatorscontdep}
Let $\{L[t,\al]: t\geq 0, \al\in \R\}$ be a family, continuous in $t$ and $\al$, of bounded linear operators $\D\mapsto \B$ ($\D$ and $\B$ are two Banach spaces equipped with a continuous inclusion $\D\to \B$) and ${U_{\al}^{t,r}}$ be backward propagators in $\B$ generated by $\{L[t,\al]\}$ on a common invariant domain $\D$
s.t. $\|U_{\al}^{t,r}\|_\B$ and $\|U_{\al}^{t,r}\|_\D$ are both uniformly bounded. Then
the family $\{U_{\al}^{t,r}\}$ depends strongly continuously on $\al$, as a family of operators in $\B$.
\end{prop}

\begin{proof}
From \eqref{PropergatorPropertya} we conclude that $U_{\al}^{t,r}f$ is a continuous function $\al \mapsto \B$ for any $f\in \D$. By the density of $\D$ in $\B$ (and boundedness of $\|U_{\al}^{t,r}\|_\B$) this property extends to all $f\in \B$.
\end{proof}

We are ready to present our results on the existence of solutions to general kinetic equations.

\begin{theorem} [local well-posedness for general anticipating case]
\label{thkineticeq}
Let $\MC$ be a convex and bounded subset of $\B^*$ with $\sup_{\mu \in \MC} \|\mu \|_{\B^*} \le M$, which is closed in the norm topologies of both $\B^*$ and $\D^*$. Suppose that

(i) the linear operators $A[t,\{\xi.\}]: \D\mapsto \B$ are uniformly bounded and Lipschitz in $\{\xi.\}$, i.e. for any $\{\xi.\}, \{\eta.\}\in C_{ \mu}([0,T], \M(\D^*))$
	\begin{equation}\label{Lip_A}
		\sup_{t\in [0,T]} \|A[t, \{\xi.\}]-A[t,\{\eta.\}]\|_{\D\mapsto \B}\leq c_1\sup_{t\in[0,T]}||\xi_t-\eta_t||_{\D^*},
	\end{equation}
\begin{equation}
\label{Lip_Aa}
\sup_{t\in [0,T]} \|A[t, \{\xi.\}]\|_{\D\mapsto \B}\leq c_1
\end{equation}
for a positive constant $c_1$;

(ii) for any $\{\xi_.\}\in C_{\mu}([0,T],\M(\D^*))$, let the operator curve $A[t,\{\xi_.\}]:\D\mapsto \B$ generate a strongly continuous backward propagator of bounded linear operators $U^{t,s}[\{\xi_.\}]$ in $\B$, $0 \leq t \leq s$, on the common invariant domain $\D$, such that
	\begin{equation}\label{BDD}
		||U^{t,s}[\{\xi.\}]||_{\D\mapsto \D}\leq c_2\,\, \text{and}\,\,||U^{t,s}[\{\xi.\}]||_{\B\mapsto \B}\leq c_3, \quad t\leq s,
	\end{equation}
for some positive constants $c_2,c_3$, and with their dual propagators $\tilde{U}^{s,t}[\{\xi_.\}]$ preserving the set $\M$.

Then, if
\begin{equation}
\label{eq0thkineticeq}
c_1c_2c_3 M T< 1,
\end{equation}
 the weak nonlinear non-Markovian Cauchy problem
\begin{equation}
\label{weak nonlinear Cauchy problem}
\frac{d}{dt}(f,\mu_t)=(A[t,\{\mu_s\}_{s\in[0,T]}]f, \mu_t),\quad \mu_0=\mu, \,\,t\in [0,T],
\end{equation}
is  well posed, that is for any $\mu\in \M$, it has a unique solution $\Phi^t(\mu) \in \M$ (that is (\ref{weak nonlinear Cauchy problem}) holds for all $f\in \D$) that depends Lipschitz continuously on time $t$ and the initial data in the norm of $\D^*$, i.e.
\begin{equation}
\label{eq2thkineticeq}
\|\Phi^s(\mu)-\Phi^t (\mu)\|_{\D^*}\leq  c_1 c_2 (s-t), \quad 0\le t \le s \le T,
\end{equation}
and
\begin{equation}
\label{eq1thkineticeq}
\begin{split}
\|\{\Phi^.(\mu)\}-\{\Phi^.(\eta)\}\|_{C([0,T],\D^*)}
&=\sup_{s\in [0,T]}\|\Phi^s(\mu)-\Phi^s(\eta)\|_{\D^*}\\
&\leq \frac{c_2}{1-c_1c_2c_3MT}\|\mu-\eta\|_{\D^*}.
\end{split}
\end{equation}
\end{theorem}

\proof

By duality
\begin{equation*}
\begin{split}
(f,(\tilde{U}^{t,0}[\{\xi_.^1\}]- \tilde{U}^{t,0}[\{\xi_.^2\}])\mu)
=((U^{0,t}[\{\xi_.^1\}]- U^{0,t}[\{\xi_.^2\}])f,\mu).
\end{split}
\end{equation*}
By Proposition \ref{prop-propergatorProperty} with $L_t^i=A[t, \{\xi^i.\}]$, together with assumptions \eqref{Lip_A} and \eqref{BDD},
\begin{equation}
\label{Lip U}
\begin{split}
&||(\tilde{U}^{t,0}[\{\xi_.^1\}]- \tilde{U}^{t,0}[\{\xi_.^2\}])\mu)||_{\B^*}\\
\leq &||U^{0,t}[\{\xi_.^1\}]- U^{0,t}[\{\xi_.^2\}]||_{\D\mapsto \B} ||\mu||_{\B^*}\\
\leq & c_1c_2c_3 tM \|\{\xi_.^1\}-\{\xi_.^2\}\|_{C([0,T],\D^*)}.
\end{split}
\end{equation}
Consequently, if \eqref{eq0thkineticeq} holds, the mapping $\{\xi_.\}\mapsto \{\tilde{U}^{t,0}[\{\xi_.\}]\}_{t\in [0,T]}$ is a contraction in $C_{\mu}([0,T],\M(\D^*))$. Hence by the contraction principle there exists a unique fixed point for this mapping and hence a unique solution to equation \eqref{weak nonlinear Cauchy problem}.

Inequality \eqref{eq2thkineticeq} follows directly from \eqref{weak nonlinear Cauchy problem}.
Finally, if $\Phi^t(\mu) = \mu_t$ and $\Phi^t(\eta) = \eta_t$, then
\begin{equation*}
\begin{split}
\mu_t-\eta_t &= \tilde{U}^{t,0}[\{\mu.\}]\mu-\tilde{U}^{t,0}[\{\eta_.\}]\eta\\
 &= (\tilde{U}^{t,0}[\{\mu.\}]-\tilde{U}^{t,0}[\{\eta.\}])\mu
  +\tilde{U}^{t,0}[\{\eta.\}](\mu-\eta).
\end{split}
\end{equation*}
From \eqref{BDD} and \eqref{Lip U},
\begin{equation}
\label{eq3thkineticeq}
\|\{\mu_.\}-\{\eta_.\}\|_{C([0,T],\D^*)} \leq  c_1c_2c_3 T M \|\{\mu_.\}-\{\eta_.\}\|_{C([0,T],\D^*)} + c_2\|\mu-\eta\|_{\D^*}
\end{equation}
implying \eqref{eq1thkineticeq}.
 \qed

\begin{remark}
Condition \eqref{eq0thkineticeq} is the analog of the condition of small
coupling (or composite gain) from \cite{HCM3}.
\end{remark}

As a consequence, we get global well-posedness for the Markovian case:

\begin{theorem}
\label{propkineticeqnonlMark}
Under the assumptions in Theorem \ref{thkineticeq}, but without the locality constraint
\eqref{eq0thkineticeq},
the Cauchy problem for kinetic equation
\begin{equation}
\label{eq1propkineticeqnonlMark}
\frac{d}{dt}(f,\mu_t)=(A[t,\mu_t]f,\mu_t),\quad \mu_s=\mu
\end{equation}
is well-posed, i.e. for any $\mu \in \MC$, $s\in [0,T]$,  it has a
unique solution $U^{t,s}(\mu) \in \MC$, $t\in [s,T]$, and the transformations $U^{t,s}$ of
$\MC$ form a propagator depending Lipschitz
continuously on time $t$ and the initial data in the norm of
$\D^*$, i.e.
\begin{equation}
\label{eqLipcontnonlinearprop}
 \|U^{t,s}(\mu)-U^{t,s}(\eta)\|_{\D^*} \le c(T) \|\mu-\eta \|_{\D^*},
\end{equation}
with a constant $c(T)$.
\end{theorem}

\begin{proof}
Global solution is constructed by extending local solutions via iterations, as is
routinely performed in the theory of ODE.
\end{proof}

It is not difficult to extend this result to non-anticipating equations yielding the following.

\begin{theorem}
\label{thkineticeqnonanticipating}
Under the assumptions in Theorem \ref{thkineticeq}, but without the locality constraint
\eqref{eq0thkineticeq},
the non-anticipating nonlinear Cauchy problem
\begin{equation}
\label{weak nonlinear Cauchy problem global NE}
\frac{d}{dt}(f,\mu_t)=(A[t,\{\mu_{\leq t}\}]f, \mu_t),\, \mu_0=\mu,\,t\geq 0,
\end{equation}
is well posed in $\M$ and its unique solution depends Lipschitz continuously on initial data in the norm of $\D^*$.
\end{theorem}

\proof
For a $\mu \in \MC$, let us construct an approximating sequence $\{\xi^n_.\} \in C([0,T],\MC(\D^*))$,
 $n=0,1,\cdots$, by defining $\xi_t^0=\mu$ for $t\in [0,T]$  and then recursively
\[
\xi^n_t=\tilde {U}^{t,0}[\xi^{n-1}_{\leq t}]\mu, \quad \forall t\in[0,T].
\]
By non-anticipation, arguing as in the proof of \eqref{Lip U} above, we first get the estimate
\[
\sup_{0\leq r\leq t}||\xi^1_{r}-\xi^0_{r}||_{\D^*}\leq c_1c_2c_3Mt,
\]
and then recursively
\[
\sup_{0\leq r\leq t}||\xi^{n}_{r}-\xi^{n-1}_{r}||_{\D^*}\leq c_1c_2c_3M
\int_0^t \sup_{0\leq r\leq s}||\xi^{n-1}_{ r}-\xi^{n-2}_{ r}||_{\D^*}\,ds
\]
that implies (by straightforward induction) that, for all $t\in [0,T]$,
\[
\|\{\xi^{n}_.\}-\{\xi^{n-1}_.\}\|_{C([0,t],\D^*)}
=\sup_{0\leq r\leq t}||\xi^{n}_{r}-\xi^{n-1}_{r}||_{\D^*} \le \frac{1}{n!}(c_1c_2c_3 Mt)^n.
\]
Hence, the partial sums on the r.h.s. of the obvious equation
\[
\xi^n_{\leq t}=(\xi^n_{\leq t}-\xi^{n-1}_{\leq t})+\cdots+(\xi^{1}_{\leq t}-\xi^{0}_{\leq t})+\xi^{0}_{\leq t}
\]
converge, and thus the sequence $\xi_{\cdot}^n$ converges in $C([0,T],\D^*)$. The limit is clearly a solution to
\eqref{weak nonlinear Cauchy problem global NE}.

To prove uniqueness and continuous dependence on the initial condition,
let us assume that $\mu_t$ and $\eta_t$ are some solutions with the initial conditions
$\mu$ and $\eta$ respectively. Instead of \eqref{eq3thkineticeq}, we now get
\[
\|\{\mu_.\}-\{\eta_.\}\|_{C([0,t],\D^*)} \leq  c_1c_2c_3 M \int_0^t \|\mu_.-\eta_.\|_{C([0,s],\D^*)}\, ds
 + c_2\|\mu-\eta\|_{\D^*}.
\]
By Gronwall's lemma, this implies
\[
\|\{\mu_.\}-\{\eta_.\}\|_{C([0,t],\D^*)} \le c_2\|\mu-\eta\|_{\D^*} e^{c_1c_2c_3 Mt}, \quad t\in[0,T]
\]
yielding uniqueness and Lipshitz continuity of solutions with respect to initial data.
\qed

For the general anticipating kinetic equations, we  have only the existence result.

\begin{theorem}[global existence of the solution for general anticipating case]
\label{thglobalexistgenkin}
Under the assumptions in Theorem \ref{thkineticeq}, but without the locality constraint
\eqref{eq0thkineticeq}, assume additionally that for any $t$ from a dense subset of $[0,T]$, the set
\begin{equation}
\label{eq1thglobalexistgenkin}
\{\tilde{U}^{t,0}[\{\xi.\}]\mu:\,\,\{\xi.\} \in C_\mu([0,T],\MC(\D^*))\}
\end{equation}
 is relatively compact in $\M$.

Then a solution to the nonlinear Cauchy problem
\begin{equation}
\label{weak nonlinear Cauchy problem-GE}
\frac{d}{dt}(f,\mu_t)=(A[t,\{\mu_s\}_{s\in[0,T]}]f, \mu_t),\quad \mu_0=\mu, t\geq 0,
\end{equation}
exists in $\M$.
\end{theorem}

\proof
Since $\M$ is convex, the space $C_{\mu}([0,T], \M(\D^*))$ is also convex. Since the dual operators $\tilde{U}^{t,0}[\{\xi.\}]$ preserve the set $\M$, for any $\{\xi.\} \in  C_{ \mu}([0,T], \M(\D^*))$, the curve $\tilde{U}^{t,0}[\{\xi.\}]\mu$ belongs to $C_{ \mu}([0,T], \M(\D^*))$ as a function of $t$.
Hence, the mapping $\{\xi.\}\to \{\tilde{U}^{t,0}[\{\xi.\}]\mu, t\in [0,T]\}$ is from $C_{ \mu}([0,T], \M(\D^*))$ to itself. Moreover, by (\ref{Lip U}), this mapping is Lipschitz continuous.

Denote $\hat{C}=\{ \{\tilde{U}^{\cdot,0}[\{\xi.\}]\mu\}:  \{\xi.\}\in C_{\mu}([0,T], \M(\D^*))\}$. Together with (\ref{eq2thkineticeq}), the assumption that set \eqref{eq1thglobalexistgenkin} is compact in $\M$ for any $t$ from a dense subset of $[0,T]$ implies that the set $\hat{C}$ is compact (Arzela-Ascoli Theorem, see A.21 in \cite{K1997}).

Finally, by Schauder fixed point theorem, there exists a fixed point in $\hat C\subset C_{\mu}([0,T], \M(\D^*))$, which gives the existence of a solution to (\ref{weak nonlinear Cauchy problem-GE}).
\qed

In case of Markov processes, the compactness assumption \eqref{eq1thglobalexistgenkin} turns out to hold under very general conditions, see below Proposition \ref{propmomentcomp} that covers all our examples.

\section{Nonlinear Markov processes in $\R^d$}
\label{secnonlinMark}

This section is designed to provide a probabilistic interpretation for nonlinear Markov evolution $\mu_t$ solving kinetic equation (\ref{eq1propkineticeqnonlMark}).
  Here we take $\B=C_{\infty}(\R^d)$ and $\M=\P(\R^d)$, so that $\B^*=\mathbf{M}^{sign}(\R)$ is the set of Borel measures on $\R^d$,
so that $M=\sup_{\mu \in \P(\R^d)} \|\mu \|_{\B^*}=1$.

Let $A[t,\mu]$, $t\geq 0$, $\mu\in \P(\R^d)$, be a family of operators in $C_{\infty}(\R^d)$ of the L\'evy-Khintchin type \eqref{fellergenerator}, that is
\begin{equation}
\label{fellergenerator 3}
\begin{split}
&A[t,\mu]f(z)=\frac{1}{2}(G(t,z,\mu)\nabla,\nabla)f(z)+ (b(t,z,\mu),\nabla f(z))\\
 &+\int (f(z+y)-f(z)-(\nabla f (z), y){\bf 1}_{B_1}(y))\nu (t,z,\mu,dy),
\end{split}
\end{equation}
such that each of them generates a Feller process with one and the same domain $\D$ such that $C^2_{\infty}(\R^d) \subseteq \D\subseteq C^1_{\infty}(\R^d)$.

\begin{definition}\label{NMP}
A nonlinear Markov process on $\R^d$ specified by a family  of generators  $\{A[t,\mu]: t\geq 0, \mu\in  \P(\R^d)\}, 0\leq s\leq t$ is a family of processes $\{X_{s,t}^{\mu}: \mu\in \P(\R^d)\}$
defined on a certain filtered probability space $(\Om, \FC, \{\FC_t\},\P(\R^d))$ that solve the nonlinear martingale problem, specified by the family $\{A[t,\mu]\}$, that is for $f\in \D$,
\begin{equation}
\label{eqdefNMP}
f(X_{s,t}^{\mu})-\int^t_s A[\tau, \LC(X_{s,\tau}^{\mu})]f(X_{s,\tau}^{\mu})d\tau, \quad s\leq t
\end{equation}
is a martingale and  $\LC(X_{s,s}^{\mu})=\mu$, where $\LC (\xi)$ denotes the distribution of a random variable $\xi$.
\end{definition}

\begin{remark}
If the operators $A$ do not depend on $t$ explicitly, i.e. $A[t,\mu]=A[\mu]$, $\{X_{s,t}^{\mu}\}$ is a time-homogeneous nonlinear Markov process.
If the operators $A$ do not depend on $\mu$, i.e. $A[t,\mu]=A[t]$, $\{X_{s,t}^{\mu}\}$ is a classical (linear) Markov process.
If the operators $A$ depend neither on $t$ nor on $\mu_t$, $\{X_{s,t}^{\mu}\}$ is a time-homogeneous Markov process.
\end{remark}

\begin{prop}\label{Prop31}

(i) For any nonlinear Markov process $\{X_{s,t}^\mu\}$, its marginal distributions $\mu_t$ solve \eqref{eq1propkineticeqnonlMark}.

(ii) Suppose the assumptions of Theorem \ref{propkineticeqnonlMark} are fulfilled with generators $A[t,\mu]$ of type (\ref{fellergenerator 3}). Then it is possible to construct a nonlinear Markov process
$\{X_{s,t}^{\mu}\}$, with $\mu_t=\LC(X_{s,t}^{\mu})$ solving the
Cauchy problem for equation \eqref{eq1propkineticeqnonlMark} with initial condition $\mu$.
\end{prop}

\proof
(i) Since \eqref{eqdefNMP} is a martingale,
\[
\E f(X_{s,t}^{\mu})=\E f(X_{s,s}^{\mu}) +\int^t_s A[\tau, \LC(X_{s,\tau}^{\mu})] \E f(X_{s,\tau}^{\mu})d\tau, \quad s\leq t.
\]
Hence, for $\mu_t=\LC(X_{s,t}^{\mu})$ we have
\[
(f,\mu_t)=(f,\mu_s)+\int_s^t (A[\tau, \mu_{\tau}]f, \mu_{\tau}) \, d\tau.
\]
Consequently, we firstly conclude that $\mu_t$ depends continuously on $t$ in $\D^*$, and secondly obtain
\eqref{eq1propkineticeqnonlMark} by differentiating this integral equation.

(ii) By the assumptions of Theorem \ref{propkineticeqnonlMark}, a solution $\mu_t\in\P(\mathbf{R}^d)$
 of equation \eqref{eq1propkineticeqnonlMark} with initial condition $\mu_s=\mu$ specifies a propagator $\tilde{U}^{t,r}[\mu_.]$, $s\leq r\leq t$,  of linear transformations in $\B^*$,
 solving the Cauchy problems for equation
\begin{equation}\label{ODE4}
	\frac{d}{dt}(f,\nu_t)=(A[t, \mu_t]f,\nu_t).
\end{equation}
In its turn, for any $\nu \in \P(\R^d)$, equation \eqref{ODE4} specifies marginal distributions of a usual (linear) Markov process $\{X_{s,t}^{\mu}(\nu)\}$ in $\mathbf{R}^d$
with the initial measure $\nu$. Clearly, the process $\{X_{s,t}^{\mu}(\mu)\}$ is a solution to our martingale problem.
\qed

\begin{remark}
\label{tangentprocess}
The proof of statement (ii) suggests that one can look at a nonlinear Markov process as a family
of solutions $\mu_t$ to a nonlinear kinetic equation, so that to each solution $\mu_t$ there is an attached usual (so-to-say, tangent) Markov process, $\{X_{s,t}^{\mu}(\nu)\}$, specified by equation \eqref{ODE4}.
In Section \ref{secepequilibrium} we shall see how these tangent processes appear in the context of interacting
particles as processes followed by tagged particles, see Remark \ref{remtaggedtangentprocess}.
\end{remark}

Using martingales allows us to prove the following useful regularity property for the solution of kinetic equations.
\begin{prop}
\label{proppmomentsfornonlin}
Suppose the assumptions of Theorem \ref{propkineticeqnonlMark} are fulfilled for a (Markovian) equation of type \eqref{eq1propkineticeqnonlMark} with generators
$A[t,\mu]$ of type (\ref{eq1propkineticeqnonlMark}). Let $\{X_{s,t}^\mu\}$ denote a nonlinear Markov process constructed from the family of generators $A[t,\mu]$ by Proposition \ref{Prop31}.
Assume, for $p \in (0,2]$ and $P>0$, the following boundedness condition holds:
\begin{equation}
\label{EBDD}
\sup_{x\in R^d,\,t\geq0, \,\mu\in\M}\max\big\{|G(t,x,\mu)|,|b(t,x,\mu)|,
\int \min (|y|^2, |y|^p) \nu(t,x,\mu,dy)\big\} \le P,
\end{equation}
and the initial measure $\mu_s=\mu$ has a finite $p$th order moment, i.e.
\[
\int |x|^p \mu (dx)=p_{\mu} <\infty.
\]

Then the distributions $\LC(X_{s,t}^{\mu}) =\Phi^{t,s}(\mu)$, solving the Cauchy problem for equation \eqref{eq1propkineticeqnonlMark} with initial condition $\mu_s$ have uniformly bounded $p$th moments, i.e.
\begin{equation}
\label{eq1proppmomentsfornonlin}
\int |x|^p \Phi^{t,s}(\mu) (dx) \le c(T,P)[1+p_{\mu}],
\end{equation}
and are
$\frac{1}{2}$-H\"older continuous with respect to $t$ in the space $(C_{Lip}(\R^d))^*$, i.e.
\begin{equation}
\label{Holder}
||\Phi^{t_1,s}(\mu)-\Phi^{t_2,s}(\mu)||_{(C_{Lip}(\R^d))^*}\leq c(T,P) \sqrt{|t_1-t_2|},
\quad \forall t_1, t_2\geq s \geq 0,
\end{equation}
with a positive constant $c$.
\end{prop}

\proof
For a fixed trajectory $\{\mu_t\}_{t\geq 0}$ with initial value $\mu$, one can
consider $\{X_{s,t}^\mu\}$ as a usual Markov process.
Using the estimates for the moments of such processes from Section 5.5 (formula (5.61) of
\cite{VK2} (more precisely, its straightforward extension to time non-homogeneous case),
one obtains from \eqref{EBDD} that
\begin{equation}
\label{eq2proppmomentsfornonlin}
\E \left[\min \big(|X_{s,t}^{\mu}-\hat x|^2, |X_{s,t}^{\mu}-\hat x|^p\big)|X_{s,s}^{\mu}=\hat x)\right]
\le e^{C(T,P)(t-s)}-1.
\end{equation}
This implies \eqref{eq1proppmomentsfornonlin}.
Moreover, \eqref{eq2proppmomentsfornonlin} implies that
 \begin{equation}
\label{eq3proppmomentsfornonlin}
\E \left[|X_{s,t}^{\mu}-\hat x| {\bf 1}_{|X_{s,t}^{\mu}-\hat x| \le 1}|X_{s,s}^{\mu}=\hat x\right[
\le (e^{C(T,P)(t-s)}-1)^{1/2}\le C(T,P) \sqrt {t-s}
\end{equation}
 \begin{equation}
\label{eq4proppmomentsfornonlin}
\E \left[|X_{s,t}^{\mu}-\hat x| {\bf 1}_{|X_{s,t}^{\mu}-\hat x| \ge 1}|X_{s,s}^{\mu}=\hat x\right]
\le (e^{C(T,P)(t-s)}-1)^{1/p} \le C(T,P) (t-s)^{1/p},
\end{equation}
and consequently
 \begin{equation}
\label{eq5proppmomentsfornonlin}
\E \left(|X_{s,t}^{\mu}-\hat x| \, |X_{s,s}^{\mu}=\hat x\right)
 \le C(T,P) \sqrt {t-s},
\end{equation}
where constants $C(T,P)$ can have different values in various formulas above.

Since $\Phi^{t,s}(\mu)$ is the distribution law of the process $\{X_{s,t}^\mu\}$,
\begin{equation}
\label{eq6proppmomentsfornonlin}
\begin{split}
||\Phi^{s,t_1}(\mu)-\Phi^{s,t_2}(\mu)||_{(C_{Lip}(\R^d))^*}
= & \sup_{||\psi||_{C_{Lip}(\R^d)}\leq 1}   \big| \E \psi(X_{s,t_1}^\mu) - \E\psi(X_{s,t_2}^\mu) \big|  \\
\leq & \sup_{||\psi||_{C_{Lip}(\R^d)}\leq 1} \E \big|\psi(X_{s,t_1}^\mu)-\psi(X_{s,t_2}^\mu)\big| \\
\leq & \E \big|X_{s,t_1}^\mu-X_{s,t_2}^\mu\big|.
\end{split}
\end{equation}
From \eqref{eq5proppmomentsfornonlin}, \eqref{eq6proppmomentsfornonlin} and Markov property, we get
\eqref{Holder} as required.
\qed

\begin{remark} For the case of diffusions, H\"older continuity \eqref{Holder} was proved in \cite{HCM3}.
\end{remark}

Our main purpose for presenting Proposition \ref{proppmomentsfornonlin} lies in the following corollary.

\begin{prop}
 \label{propmomentcomp}
 Under the assumptions of Theorem \ref{thkineticeq} for generators
$A[t,\{\mu_.\}]$ of L\'evy-Khintchin type, but without locality condition
\eqref{eq0thkineticeq}, suppose the boundedness condition \eqref{EBDD} holds for some $p \in (0,2]$ and $P>0$.
Then the compactness condition from Theorem \ref{thglobalexistgenkin} (stating that set
\eqref{eq1thglobalexistgenkin} is compact in $\PC(\R^d)$) holds for any initial measure $\mu$
with a finite moment of $p$th order.
 \end{prop}

 \begin{proof}
 It follows from  \eqref{eq1proppmomentsfornonlin} and an observation that a set of probability laws on $\R^d$ with a bounded $p$th moment, $p>0$, is tight and hence relatively compact.
 \end{proof}

\section{Basic examples in $\R^d$ and $\R^d\times\{1,\cdots,K\}$}
\label{examples}

In this section, we present some basic examples of generators that fit to assumptions of the abstract theorems in Section \ref{secabstractnonMark}.

Notice that for the study of Markovian kinetic equations (without a coupling with control) the most nontrivial condition of Theorem \ref{thkineticeq} is (ii), as it concerns the difficult question from the theory of usual Markov process, on when
a given pre-generator of L\'evy-Khintchine type does really generate a Markov process.
More difficult is the situation with time-dependent generators, as the standard semigroup methods
(resolvents and Hille-Phillips-Iosida theorem) are not applicable. Further discussion of this topic and other examples
can be found in literature suggested in Section \ref{secconclud}.

\begin{remark} For general anticipating equations coupled with control, assumption (i) of Theorem \ref{thkineticeq}
(yielding feedback regularity) becomes nontrivial as well. Verifiable conditions for this assumption to hold are given at the end of Section \ref{Sensitivity} as the main consequence of the sensitivity theory for HJB equation developed there.
\end{remark}

\begin{example} Nonlinear L\'evy processes are specified by a families of generators of type (\ref{fellergenerator 3}) such that all coefficients do not depend on $z$, i.e.
\begin{equation}
\begin{split}
A[t,\mu]f(x)=&\frac{1}{2}(G(t,\mu) \nabla,\nabla)f(x)+(b(t,\mu), \nabla f)(x)\\
&+\int [f(x+y)-f(x)-(y, \nabla f(x))\mathbf{1}_{B_1}(y)]\nu(t,\mu,dy). \nonumber
\end{split}
\end{equation}

The following statement is a consequence of Proposition 7.1 from \cite{Ko10} (and can be easily obtained from the general theory of processes with independent increments as developed e.g. in \cite{GiSk75}).

\begin{prop}
\label{NLP}
Supposed that the coefficients $G,b,\nu$ are continuous in $t$ and Lipschitz continuous in $\mu$ in the norm of Banach space $(C^2_{\infty}(\R^d))^*$, i.e.
\begin{equation*}
\begin{split}
\|G(t,\mu)-G(t,\eta)\|+&\|b(t,\mu)-b(t,\eta)\|+\int\min(1,|y|^2)|\nu(t,\mu,dy)-\nu(t,\eta, dy)|\\
\leq &c||\mu-\eta||_{(C^2_{\infty}(\R^d))^*},\quad \forall t\geq 0,
\end{split}
\end{equation*}
with a positive constant $c$, then condition (ii) of Theorem \ref{thkineticeq} holds with $\D =C^2_{\infty} (\R^d)$.
\end{prop}
\end{example}

\begin{example}
\label{MV Diffusion}
McKean-Vlasov diffusion are specified by the following stochastic differential equation
\[
dX_t = b(t, X_t, \mu_t)dt + \sigma(t, X_t, \mu_t)dW_t,
\]
where drift coefficient $b:\R^+\times \mathbf{R}^d \times \P(\mathbf{R}^d)\rightarrow \mathbf{R}^d$, diffusion coefficient $\sigma: \R^+\times\mathbf{R}^d \times \P(\mathbf{R}^d)\rightarrow \mathbf{R}^d$ and $W_t$ is a standard Brownian motion. The corresponding generator is given by
\[
A[t,\mu]f(x) = (b(t, x,\mu), \nabla f(x))+ \frac{1}{2}(G(t, x,\mu), \nabla^2f(x)),\quad f\in C^2_{\infty}(\R^d),
\]
where $G(t, x,\mu)=tr\{\sigma (t, x,\mu)\sigma^T (t, x,\mu)\}$.
It is well known (and follows from Ito's calculus) that if the coefficients of a diffusion are
Lipshitz continuous, the corresponding SDE is well posed, implying the following.

\begin{prop}
\label{propnonlindif}
If $G, b$ are continuous in $t$, Lipshitz continuous in $x$ and Lipschitz continuous in $\mu$ in the topology of $(C^2_{\infty}(\R^d))^*$,
then the condition (ii) of Theorem \ref{thkineticeq} is satisfied.
\end{prop}
\end{example}

\begin{example} Nonlinear stable-like processes (including tempered ones) are specified by the families
\begin{equation}
\begin{split}
A[t,\mu]f(x) = &(b(t,x,\mu), \nabla f(x))+\int (f(x+y)-f(x)) \nu(t,x,\mu,dy) \\
&+ \int_0^K d|y| \int_{S^{d-1}}
 a(t,x,s)\frac {f(x+y)-f(x)-(y,\nabla f(x))}{|y|^{\al_t(x,s)+1}}
 \om_t(ds).
\end{split}
\end{equation}
Here $s=y/|y|$, $K>0$, $\om_{t}$ are certain finite Borel measures on
$S^{d-1}$ and $\nu(t,x,\mu,dy)$ are finite measures, $a$, $\al$ are positive bounded functions with
$\al \in (0,2)$.

The following result is a corollary of (a straightforward time-nonhomogeneous extension of) Proposition
4.6.2 of \cite{VK2}.
\begin{prop}
\label{propstablelikesmooth}
If all coefficients are continuous in $t$, $a,\al$ are $C^1$-functions in $x,s$, $b$ and $\nu$ are Lipshitz continuous in $x$ and $\mu$ (with $\mu$ taken in the topology $(C^2_{\infty}(\R^d))^*$,
then all conditions of Theorem \ref{thkineticeq} are satisfied with $\D =C^2_{\infty} (\R^d)$.
\end{prop}
\end{example}

\begin{example}
 \label{OAMO}
Processes of order at most one are specified by the families
\begin{equation}\label{eqshortgeneratornonlin}
 A[t,\mu]f(x)=(b(t,x,\mu),\nabla f(x))+\int_{\mathbf{R}^d}(f(x+y)-f(x))\nu (t,x,\mu,dy), \nonumber
\end{equation}
with the L\'evy measures $\nu$ having finite first moment $\int |y|\nu(t,x,\mu,dy)$.
The next result is established in Theorem 4.17 of \cite{Ko10}.

\begin{prop}
If $b,\nu$ are continuous in $t$ and Lipschitz continuous in $\mu$, i.e.
\begin{equation*}
\begin{split}
||b(t,x,\mu)-b(t,x,\eta)||+&\int |y|\nu(t,x,\mu,dy)-\nu(t,x,\eta, dy)|\\
\leq &c||\mu-\eta||_{(C^1_{\infty}(\R^d))^*},\quad \quad \forall t\geq 0, x\in \R^d
\end{split}
\end{equation*}
and Lipshitz continuous in $x$,
then condition (ii) of Theorem \ref{thkineticeq} is satisfied with $\D =C^1_{\infty} (\R^d)$.
\end{prop}
The generators of order at most one describe a variety of well known models including spatially homogeneous and
mollified Boltzmann equation and interacting $\al$-stable laws with $\al <1$.
\end{example}

\begin{example}
\label{integralform} Pure jump processes on a locally compact metric $\XC$ are specified by integral generators of the form
\begin{equation}
A[t,\mu]f(x) = \int_{\R^d}(f(y)-f(x))\nu(t, x,\mu,dy)
\end{equation}
with bounded measures $\nu$. Assuming, unlike our basic assumption \eqref{fellergeneratorwithdriftcont},
that we can control jumps of this process, i.e. measure $\nu$ depends on control $u$, basic kinetic equation
\eqref{eqkineqmeanfieldnonM} takes the form
\begin{equation}
\label{eqkineqmeanfieldnonMinteg}
 {d \over dt} (f, \mu_t)
 =\int \int (f(y)-f(x)) \nu (t, x,\mu_t, \Gamma (t,x,\{\mu_{\ge t}\}), dy) \mu_t (dx),
\end{equation}
or in the strong form
\begin{equation}
\label{eqkineqmeanfieldnonMinteg1}
\begin{split}
 {d \over dt}\mu_t (dx)
 =\int &\nu (t, z,\mu_t, \Gamma (t,z,\{\mu_{\ge t}\}), dx) \mu_t (dz)\\
& - \mu_t (dx) \int  \nu (t, x,\mu_t, \Gamma (t,x,\{\mu_{\ge t}\}), dz),
 \end{split}
\end{equation}
and the corresponding HJB equation becomes
 \begin{equation}
\label{eqkineqmeanfieldnonMintegHJB}
 {\pa V \over \pa t} (t,x)+\max_u \left[J(t,x,\mu_t,u)+ \int (V(y)-V(x)) \nu (t, x,\mu_t, u, dy)\right]=0.
\end{equation}
In particular, if $\XC$ is a finite set $\{1,\cdots, K\}$, measures become $K$-dimensional vectors and kinetic
equation \eqref{eqkineqmeanfieldnonMinteg1} rewrites as
\begin{equation}
\label{eqkineqmeanfieldnonMinteg2}
 \dot \mu_t^i
 =\sum_j \nu (t,j,\mu_t, \Gamma (t,j,\{\mu_{\ge t}\}),i) \mu_t^j
 - \mu_t^i \sum_j  \nu (t,i,\mu_t, \Gamma (t,i,\{\mu_{\ge t}\}), j).
\end{equation}
It is easy to see that, if measures $\nu (t, x,\mu_t, u, dy)$ are uniformly bounded,
the conditions of Theorem \ref{thkineticeq} are satisfied with $\D =C_{\infty} (\R^d)$.
For unbounded rates we refer to \cite{Ko10} (and references therein) for a detailed discussion.
\end{example}

Finally, we are interested in a multi agent setting of Subsection \ref{subescdiscrete}.

\begin{theorem}
\label{propmainmultiagentexamples}
Let $A=(A_1,\cdots, A_K)$ be a family of operators in $C_{\infty}(\R^d)$ of form
\eqref{fellergeneratorwithdriftcont}, with $u=\Gamma (t,z, \{\mu_{\ge t}\})$, where $L_i$ are given by one of the examples above, with all required continuity assumptions holding for each $A_i$. Moreover, let $h_i$ be a Lipshitz continuous function of $u$ and $\Gamma$ depends Lipshitz continuously on all its arguments with measures considered in the topology of $\D^*$.

(i) Then all conditions and hence the conclusions of Theorem \ref{thkineticeq} are satisfied.

(ii) If boundedness conditions of Proposition \ref{propmomentcomp} hold for each $A_i$, $i=1,\cdots, K$,
then the compactness requirement of Theorem \ref{thglobalexistgenkin} holds for any initial measure $\mu$
with a finite moment of $p$th order and hence the corresponding existence result holds.
\end{theorem}

\begin{proof}
Looking at the proofs of all above statements, one sees that they extend straightforwardly to this multi-agent setting.
\end{proof}


\section{Cauchy problem for HJB equations}
\label{Cauchy HJB}

In this section, we are concerned with the general well-posedness of the Cauchy problem for the HJB equation (\ref{HJB}). More precisely, we shall prove that there exists a unique mild solution $V(t,x)$ to the Cauchy problem
\begin{gather}
\label{general HJB}
\begin{cases}
-\frac{\partial{V}}{\partial{t}} = H_t(x,\nabla V) +L_tV\\
V|_{t=T}=V^T(\cdot)
\end{cases}
\end{gather}
with rather general Hamiltonian $H_t$, assuming that the propagator generated by $L_t$ has certain smoothing property. Here we mostly follow Chapter 7 in \cite{VK2}, preparing the setting for the sensitivity analysis of the next section.

\begin{remark}
In the control theory , Hamiltonians appear in form \eqref{HJB 2a}, i.e. as
\begin{equation}
\label{HJB 2b}
H_t(x,p)=\max_{u\in\U} ( h(t,x,u)p+J(t,x,u)).
\end{equation}
where $u$ are controls, $x,p\in \mathbf{R^d}$ and $t>0$.
\end{remark}

It is well known and easy to show (by Duhamel's principle) that if $V(t,x)$ is a classical (smooth) solution to
\eqref{general HJB}, then it also satisfies the following {\it mild form} of \eqref{general HJB}.

\begin{equation}
\label{mild_value_function}
V(t,x) = (U^{t,T} V^T(\cdot))(x) + \int_t^T U^{t,s}H_s(\,\cdot\,,\nabla V(s,\cdot))(x) ds.
\end{equation}
 This integral form is usually easier to analyze than the original equation \eqref{general HJB} and is sufficient for applications to optimal control problems. Hence we shall deal here with the well-posedness of the mild equation.

\begin{theorem}
\label{weak_existence}
Suppose

(1) $H_t(x,p)$ is continuous in t, Lipschitz continuous in $p$ uniformly in x, i.e.
\begin{equation}
\label{eq1thweak_existence}
\|H_t(x,p)-H_t(x,p')\|\leq c_1\|p-p'\|, \quad p,p'\in \mathbf{R}^d,\, \forall x\in \mathbf{R}^d, t\in[0,T]
\end{equation}
with a positive constant $c_1$ and there exists a constant $h>0$ such that
\begin{equation}
\label{eq2thweak_existence}
||H_t(x,0)||\leq h\quad \forall x\in \mathbf{R}^d,\quad t\in[0,T];
\end{equation}

(2) the operators $L_t$ generate a strongly continuous backward propagator ${U}^{t,s}$ in the Banach space $C_{\infty}(\R^d)$ with a common invariant domain $\D\subset C^1_\infty(\R^d)$, and the subspace $C_{\infty}^1(\R^d)$ is also invariant, so that $U^{t,s}$ are also strongly continuous in $C^1_{\infty}$ and
\begin{equation}
\label{eq3thweak_existence}
\|{U}^{t,s}\phi\|_{C_{\infty}^1(\R^d)} \leq c_2, \quad \|{U}^{t,s}\phi\|_{C_{\infty}(\R^d)} \leq c_2
\end{equation}
for all $t\leq s\leq T$ and a constant $c_2>0$.

(3) the operators ${U}^{t,s}$ map $C_{\infty}(\mathbf{R}^d)$ to $C_{\infty}^1(\mathbf{R}^d)$ for $t<s$, and
\begin{equation}
\label{smooth property}
\|{U}^{t,s}\phi\|_{C_{\infty}^1(\mathbf{R}^d)} \leq w(s-t)\|\phi\|_{C_{\infty}(\mathbf{R}^d)},
\end{equation}
for $t<s<T$ and for all $\phi\in C_{\infty}(\mathbf{R}^d)$, with an integrable positive function $w$ on $[0,T]$.

Then for any terminal data $V^T(\cdot)\in C_{\infty}^1(\mathbf{R}^d)$, there exists a unique solution $V(t,x)$ to mild equation (\ref{mild_value_function}), which is of class $C_{\infty}^1(\mathbf{R}^d)$ for all $t$.
\end{theorem}

\proof
We shall write sometimes shortly $C_{\infty}^1$ and $C_{\infty}$ for $C_{\infty}^1(\R^d)$ and $C_{\infty}(\R^d)$ respectively when no confusion may arise.
 Define an operator $\Psi$ acting on $C_{V^T}^T([0,T], C_{\infty}^1(\R^d))$ by the formula
\begin{equation}
\label{mildformpsi1}
\Psi_t[\phi](x)= (U^{t,T} V^T)(x) + \int_t^T U^{t,s}H_s(\cdot,\nabla \phi_s)(x)ds
\end{equation}
We claim that $\Psi$ maps $C_{V^T}^T([0,T], C_{\infty}^1(\R^d))$ to itself.

By the triangle inequality and \eqref{eq1thweak_existence},
\begin{eqnarray}
\|H_t(x,\nabla \phi_t)\|_{C_{\infty}} &\leq& \|H_t(x,0)\|_{C_{\infty}}
+\|H_t(x,\nabla \phi_t)-H_t(x,0)\|_{C_{\infty}} \nonumber \\
              &\leq& h+c_1\|\nabla \phi_t \|_{C_{\infty}} \nonumber \\
              &\leq & h+c_1\| \phi_t \|_{C_{\infty}^1}. \nonumber
\end{eqnarray}
From assumptions (2)-(3),
\begin{equation}
\label{eq4thweak_existence}
\begin{split}
\|\Psi(\phi_t)\|_{C_{\infty}^1} & \leq \|{U}^{t,T}V^T(\cdot)\|_{{C_{\infty}^1}}+\int_t^T \| U^{t,s} H_t(\cdot,\nabla \phi_s)\|_{C_{\infty}^1} ds \\
&\leq c_2 \|V^T \|_{C_{\infty}^1} + \int_t^T w(s-t)\|H_s(\cdot,\nabla \phi_s)\|_{C_{\infty}} ds\\
& \leq c_2 \|V^T \|_{C_{\infty}^1} + (h+c_1\sup_{t\leq s\leq T}\|\phi_s\|_{C_{\infty}^1})\int_t^T w(s-t)ds.
\end{split}
\end{equation}
Hence
\[
\Psi: C_{V^T}^T([0,T], C_{\infty}^1(\R^d))\mapsto C_{V^T}^T([0,T], C_{\infty}^1(\R^d)).
\]

Furthermore, by equation \eqref{eq1thweak_existence} and \eqref{smooth property},
\begin{equation}
\label{eq5thweak_existence}
\begin{split}
\|\Psi_t(\phi^1_.)-\Psi_t(\phi^2_.)\|_{C_{\infty}^1}
&\leq \int_t^T \| U^{t,s}[H_s(\cdot,\nabla \phi^1)-H_s(\cdot,\nabla \phi^2)]\|_{C_{\infty}^1}ds\\
&\leq \int_t^T c_1w(s-t)\|\nabla \phi_s^1-\nabla \phi_s^2 \|_{C_{\infty}}ds\\
&\leq c_1\sup_{t\leq s\leq T} \| \phi_s^1-\phi_s^2 \|_{C_{\infty}^1}\int_t^T w(s-t)ds
\end{split}
\end{equation}
for $\phi^1,\phi^2 \in C_{V^T}^T([0,T], C_{\infty}^1(\R^d))$. So the mapping $\Psi_t$ is a contraction in $C_{V^T}^T([0,T], C_{\infty}^1(\R^d))$ for time $t\in [T-t_0, T]$ with $t_0>0$ small enough. Consequently $\Psi$ has a unique fixed point for $t\in [T-t_0, T]$. The well-posedness on the whole interval $[0,T]$ is proved, as usual, by iterations. \qed

\begin{remark}
If each operator $U^{t,s}$ has a kernel, e.g. it is given by
\begin{equation}
U^{t,s} f(x)= \int G(t,s,x,y) f(y) dy
\end{equation}
with a certain Green's function $G$, such that
\begin{equation}
\| \nabla_x G(t,s,x,\cdot) \|_{L^1(\R^d)}\leq w(s-t)
\end{equation}
then the smoothing condition (\ref{smooth property}) holds.
\end{remark}

By the well posedness of equation \eqref{mild_value_function}, its solutions define a propagator
in $C^1_{\infty}(\R^d)$. This implies, by the standard results (see e.g. \cite{FlSo}),
that these solutions are viscosity solutions to the original equation \eqref{general HJB}.
Moreover, for $H$ of form \eqref{HJB 2b}, these viscosity solutions solve the corresponding optimization problem.
That is all we need for our purposes, allowing us to work further only with mild solutions constructed above.

\begin{remark}
 If we want a solution of mild equation (\ref{mild_value_function}) to solve the original
 HJB equation (\ref{general HJB}), one has to make some additional regularity assumptions,
 see e.g. Theorem 7.8.2 of \cite{VK2} for detail.
\end{remark}

We would like to mention that Example \ref{OAMO} and Example \ref{integralform}  in Section \ref{examples} do not fit into our setting for the well-poseness of HJB equation \eqref{general HJB}, since the smoothing property \eqref{smooth property} is hard to ensure for this examples. In the present paper we shall not develop an appropriate framework for HJB equation that would fit to these examples.

Our main examples are summarized as follows.

\begin{prop}
\label{proponsmoothing}
(i) Under the assumptions of Propositions \ref{NLP} and \ref{propnonlindif} suppose the matrix $G$ is uniformly positive
and is $C^1$ as function $x$ in the second case. Then the corresponding propagator (for any fixed $\{\mu_t\}_{t\in[0,T]}$)
has the smoothing property \eqref{smooth property}.
(ii) Under the assumptions of Propositions \ref{propstablelikesmooth}, suppose the measure $\om$ has an everywhere positive density with respect to the Lebesgue measure on $S^{d-1}$. Then the corresponding propagator (for any fixed $\{\mu_t\}_{t\in[0,T]}$)
has the smoothing property \eqref{smooth property}.
\end{prop}

\begin{proof}
Statement (i) is well known in the theory of diffusion processes. Statement (ii) is proved in
\cite{Ko00}, see also \cite{VK2}.
\end{proof}


\section{Sensitivity analysis for HJB equations}
\label{Sensitivity}

For using nonlinear Markov processes as a modelling tool in concrete problems (e.g. mean field games), it is important to be able to assess how sensitive the behaviour of the process is when the key parameters of the model change. Ideally one would like to have some kind of smooth dependence of evolution on these parameters. In this section, we shall discuss the sensitivity of solutions of HJB equations.

\begin{theorem} \label{Sensitivity 1}
Suppose a continuous mapping $(t,\alpha)\mapsto L[t,\alpha]$ from $[0,T]\times \R$ to bounded linear operators $L[t,\alpha]:\D\mapsto C_{\infty}(\mathbf{R}^d)$ is given, $\D \subset C^1_{\infty}(\R^d)$, such that

(i) the linear operators $L[t,\alpha]$ are differentiable with respect to $\alpha$, i.e. $\partial L/\partial \alpha[t,\alpha]$ exist and represent a continuous (in $t,\al$) family of uniformly bounded (by a constant $c_1>0$) operators $\D \to C_{\infty}(\mathbf{R}^d)$;

(ii) for any $\alpha \in \R$, the operator curve $L[t,\al]:\D\rightarrow C_{\infty}(\mathbf{R}^d)$ generates a strongly continuous backward propagator of bounded linear operators $U^{t,s}_\al$ in the Banach space $C_{\infty}(\mathbf{R}^d)$ with the common invariant domain $\D$, such that,
for any $\al$, and $t\leq s$
\begin{equation}\label{BDD 2}
		\max \{||U^{t,s}_\al||_{\D},||U^{t,s}_\al||_{C_\infty}, ||U^{t,s}_\al||_{C^1_\infty}\}\leq \textbf{K}
	\end{equation}
for some constant $\textbf{K}> 0$, and with their dual propagators $\tilde{U}^{s,t}_\al$ preserving the set $\PC(\R^d)$.

(iii) the operators $U^{t,s}_\al: C_{\infty}(\mathbf{R}^d)\mapsto C^1_{\infty}(\mathbf{R}^d)$, $0 \leq t < s$, and
\begin{equation}
\label{smooth property 2}
\|{U}^{t,s}_\al\phi\|_{C_{\infty}^1(\mathbf{R}^d)} \leq w(s-t)\|\phi\|_{C_{\infty}(\mathbf{R}^d)},
\end{equation}
for $t<s<T$ and for all $\phi\in C_{\infty}(\mathbf{R}^d)$, with an integrable positive function $w$ on $[0,T]$.

(iv) for any $\al\in\R$, $V_\alpha^T\in \D$ and $\partial V_\alpha^T/\partial \alpha$ exists in $C_{\infty}^1(\mathbf{R}^d)$ and represents a continuous mapping $\al\mapsto C_{\infty}^1(\mathbf{R}^d)$.

Then, for
\begin{equation}
\label{smooth property 3}
V_{\al}(t,\cdot)=U^{t,T}_{\al} V^T_{\al}(\cdot)\in\D, \quad \forall t\in[0,T]
\end{equation}
the derivative $\partial{V_\alpha}/\partial{\alpha}(t,\cdot)$ exists in the topology of the Banach space $C^1_{\infty}(\mathbf{R}^d)$ and is uniformly bounded in $C^1_{\infty}(\mathbf{R}^d)$.
\end{theorem}

\proof For notational simplicity, denote
$$V'_\alpha (t,x)=\frac{\partial V_\alpha}{\partial \alpha}(t,x).$$
Differentiating both sides of the equation
$$\frac{\partial{V_\alpha}}{\partial{t}}(t,x)=- L[t,\alpha]V_\alpha (t,x)$$
with respect to $\alpha$ yields
$$\frac{\partial}{\partial t}V'_\alpha (t,x) = - L[t,\alpha]V'_\alpha (t,x)-\frac{\partial{L}}{\partial{\alpha}}[t, \alpha]V_\alpha (t,x).$$

Then by the Duhamel principle,

\begin{equation}
\label{smooth property 4}
V'_\alpha (t,x) = {U}_\alpha^{t,T} (V^T_\alpha)' (x) +\int_t^T {U}^{t,s}_{\alpha}(\frac{\partial{L}}{\partial{\alpha}}[s,\alpha]V_\alpha )(s,x)ds.
\end{equation}
Since $(V_\alpha^T)'(\cdot)\in C_{\infty}^1(\mathbf{R}^d)$ and the family ${U}^{t,T}_{\alpha}$ is bounded as a family of  mappings from $C_{\infty}^1(\mathbf{R}^d)$ to $C_{\infty}^1(\mathbf{R}^d)$, ${U}^{t,T}_{\alpha}V_\alpha'(T,\cdot)$ belongs to  $C_{\infty}^1(\mathbf{R}^d)$ and is uniformly bounded for bounded $t, T$.

Also since for each pair $(t,\alpha)\in [0,T]\times \R$,  the operator $\partial A/\partial \alpha[t,\alpha]:\D\mapsto C_{\infty}(\mathbf{R}^d)$ and for $0< t< s< T$, the operator ${U}^{t,s}_{\alpha}: C_{\infty}(\mathbf{R}^d)\mapsto C_{\infty}^1(\mathbf{R}^d)$, we have
\[
{U}^{t,s}_{\alpha}(\frac{\partial{L}}{\partial{\alpha}}[s,\alpha]V_\alpha )(s,x)\in C_{\infty}^1(\mathbf{R}^d).
\]

From the assumptions \eqref{BDD 2} and \eqref{smooth property 2},  the following inequality
\begin{equation}
\begin{split}
||V'_\alpha (t,\cdot)||_{C_{\infty}^1} \leq& || {U}_\alpha^{t,T} (V^T_\alpha)' (\cdot)||_{C_{\infty}^1} \\
    &+\int_t^T ||{U}^{t,s}_{\alpha}(\frac{\partial{L}}{\partial{\alpha}}[s,\alpha]V_\alpha )(s,\cdot)||_{C_{\infty}^1}ds\\
    \leq &  \textbf{K} \|V^T\|_{C^1_{\infty}}+\int_t^Tw(s-t)\|\frac{\partial{L}}{\partial{\alpha}}[s,\alpha]V_\alpha (s,\cdot)\|_{C_\infty(\R^d)}ds\\
    \leq &(\textbf{K}+c_1\int_t^Tw(s-t)ds) \|V^T\|_{\D}
\end{split}
\end{equation}
gives us the uniform boundedness of the derivative $V'_\alpha (t,\cdot)$ in $C_{\infty}^1(\mathbf{R}^d)$,
showing that $V'_\alpha (t,\cdot) \in C_{\infty}^1(\mathbf{R}^d)$.

Finally, using smoothing property \eqref{smooth property 2} in conjunction with the arguments of Proposition
\ref{prop-propergatorscontdep}, we can strengthen the statement of Proposition \ref{prop-propergatorscontdep}
by obtaining continuous dependence of ${U}^{t,T}_{\alpha}$ not only as operators in $C_{\infty}(\mathbf{R}^d)$,
but also as operators in $C_{\infty}^1(\mathbf{R}^d)$. This implies the
continuity of function \eqref{smooth property 4},
as a mapping $\al \mapsto C_{\infty}^1(\mathbf{R}^d)$. Together with our uniform bounds on \eqref{smooth property 4},
this implies that this function does in fact represent the derivative of $V_\al(t,x)$ with respect to $\al$.
  \qed

Differentiating \eqref{smooth property 3} with respect to $\al$ formally, we get
\begin{equation}
\label{smooth property 5}
V'_\alpha (t,.) = {U}_\alpha^{t,T} (V^T_\alpha)' (.)
+\left(\frac{\partial}{\partial \alpha}U_\alpha^{t,T}\right)(V^T_\alpha).
\end{equation}
Our proof above and in particular formula \eqref{smooth property 4} can be interpreted as saying that
the derivative $\partial U_\alpha^{t,T}/\partial \alpha $ is well defined as a bounded operator
$\D\to C_{\infty}^1(\mathbf{R}^d)$ such that
\begin{equation}
\label{smooth property 6}
\frac{\partial}{\partial \alpha}U_\alpha^{t,T}
=\int_t^T {U}^{t,s}_{\alpha} \frac{\partial{L}}{\partial{\alpha}}[s,\alpha] U_\al^{s,T} \, ds.
\end{equation}

\begin{theorem} \label{Sensitivity 2}

Under the assumptions of Theorem \ref{Sensitivity 1},
suppose a family of Hamiltonian functions $H_{\al,t}(x,p)$ is given that satisfy, for each $\al \in \R$,
all the assumptions of Theorem \ref{weak_existence}, with all estimates being uniform in $\al$.
Assume moreover that the derivative $\pa H_{\al,t}/\pa \al (x, p)$ exists and is a continuous function of all its variables.

Denote by $V_\alpha (t,x)$ the mild solution of the Cauchy problem
\begin{gather}
\begin{cases}
\frac{\partial{V_\alpha}}{\partial{t}}(t,x) =- L[t,\alpha]V_\alpha (t,x)-H_{\al,t}(x, \nabla V_\alpha(t,x))\\
V_\alpha|_{t=T}=V_\alpha^T(\cdot),
\end{cases}
\end{gather}
constructed in Theorem \ref{weak_existence}.
Then  the derivative $\partial{V_\alpha}/\partial{\alpha}(t,x)$ exists  and is uniformly bounded in the Banach space $C^1_{\infty}(\mathbf{R}^d)$ (as a function of $x$).
\end{theorem}

\proof
For any $\alpha \in \R$, by Theorem \ref{weak_existence}, the unique solution $V_\alpha (t,x)$ is the unique fixed point of the mapping $\phi \mapsto \Psi_{\al}(\phi)$ defined by
\begin{equation}
\label{mapping Psi}
\Psi^t_{\al}(\phi) = U_\alpha^{t,T} V_\alpha^T + \int_t^T U_\alpha^{t,s}H_{\alpha,s}(.,\nabla \phi_s (.))ds.
\end{equation}
Differentiating with respect to $\al$ the fixed point equation $V_{\al}=\Psi_{\al} (V_\alpha )$ and taking into
account \eqref{smooth property 6} we get the following equation
\begin{equation}
\label{eqdiffixedpoint}
V'_{\al}(t,.)
=\Om^t_\al + \digamma^t_\al (V_{\al}')
\end{equation}
with
\begin{equation}
\label{eqdiffixedpoint1}
\digamma^t_\al (\varphi)=\int_t^T U_\alpha^{t,s} \frac{\pa H_{\al,s}}{\pa p}(., \nabla V_\alpha(s,.)) \nabla \varphi(s,.) \, ds,
\end{equation}
\[
\Om^t_\al=\frac{\partial}{\partial \alpha}\left(U_\alpha^{t,T} V_{\al}^T\right)
+\int_t^T \left(\frac{\partial}{\partial \alpha}U_\alpha^{t,s} \right)H_{\al,s}(., \nabla V_\alpha(s,.)) ds
\]
\[
+\int_t^T U_\alpha^{t,s} \frac{\pa H_{\al,s}}{\pa \al}(., \nabla V_\alpha(s,.)) ds.
\]
Hence formally,
\begin{equation}
\label{eqdiffixedpoint2}
V_{\al}'=(Id-\digamma_\al)^{-1} \Om_\al,
\end{equation}
where $Id$ stands for the identity operator.
By \eqref{eqdiffixedpoint1},
\begin{equation}
\label{eqdiffixedpoint3}
\begin{split}
\sup_{s\in [t,T]} \|\digamma_s (\varphi)\|_{C^1_{\infty}} &\le \int_t^T w(s-t) \sup_{\al,s,z,p}
\left| \frac{\pa H_{\al,s}}{\pa p}(z,p)\right| \sup_{s\in [t,T]} \|\varphi_s\|_{C^1_{\infty}} \, ds\\
&\le c(t) \sup_{s\in [t,T]} \|\varphi_s\|_{C^1_{\infty}}
\end{split}
\end{equation}
with $C(t) \to 0$ as $t\to T$. Hence the inverse operator \eqref{eqdiffixedpoint2} is well defined
for $t$ close enough to $T$ yielding uniform boundedness of the l.h.s. of \eqref{eqdiffixedpoint2}. Together
with the continuity of $V_{\al}'$ in $\al$, this allows us to conclude that the derivative $V_{\al}'(t,.)$ exists in
$C^1_{\infty}$ for all $t$ close to $T$. As usual, this statement extends to all $t$ by iterations.
\qed

\begin{remark} Assuming some continuity in $t$ of the operators
 $U^{t,s}_\al: C_{\infty}(\mathbf{R}^d)\mapsto C^1_{\infty}(\mathbf{R}^d)$, one could further claim that
$V_{\al}'(t,.)$ is also continuous as mapping $(\al, t)\mapsto C^1_{\infty}(\mathbf{R}^d)$, but this does not seem to be crucial.
\end{remark}

We want now to apply these results to our basic mean field setting (in particular that of Subsection \ref{subescdiscrete}), that is to HJB equation \eqref{HJB 2}, \eqref{HJB 2a}:
\begin{equation}
\label{HJB 20}
\frac{\partial V^i (t,x)}{\partial t}+ H_t^i(x,\nabla V^i (x), \mu_t)+ L_i[t,\mu_t]V^i(t,x)=0,
\end{equation}
\begin{equation}
\label{HJB 20a}
H_t^i(x,p, \mu_t):=\max_{u\in\U} \{  h_i(t,x,\mu_t,u)p+J_i(t,x,\mu_t,u)\}.
\end{equation}

This leads to the following crucial result.

\begin{theorem}
\label{thfeedbackHJB}
Under the assumptions for $H^i$ and $L_i$ from Theorem \ref{Sensitivity 2}
(recall that basic examples when smoothing property \eqref{smooth property 2} holds are given in Proposition \ref{proponsmoothing}), assume additionally

(i) $\frac{\delta H^i}{\delta\mu(\cdot)}$ exists and is continuous and uniformly bounded;

(ii) $\frac{\delta L_i}{\delta\mu(\cdot)}$ exists in the topology of $\D^*=(C^2_{\infty}(\mathbf{R}^d))^*$ and represents a continous family of uniformly bounded operators $\D\to C_\infty(\R^d)$;

(iii) the unique point of maximum in \eqref{HJB 20a} is continuous in $t$ and Lipschitz continuous in $(x,p,\mu)$, uniformly with respect to $t$, $x$ and bounded $p$ and $\mu$ (again $\mu$ in the topology of $\D^*$).

Then, given a trajectory $\{\mu.\}\in C_\mu([0,T], \MC(\D^*))$, $\MC=\P(\R^d)$, and a final payoff $V^T$, the feedback control $u=\Gamma(t,x,\{\mu_{\geq t }\})$ defined via equations \eqref{HJB 20} and \eqref{HJB 20a}, is  Lipschitz continuous in $\{\mu.\}$ uniformly, i.e. for any $\{\eta.\},\, \{\xi.\} \in C_\mu([0,T], \MC(\D^*))$
\[
|| \Gamma(t,x, \{\eta_{\geq t}\}-\Gamma(t,x, \{\xi_{\geq t}\}||\leq   k_1\sup_{s\in [t,T]}||\eta_s -\xi_s||_{\D^*}, \quad \forall t\in [0,T],\,x\in \R^d
\]
with some constant $k_1>0$.
\end{theorem}

\proof
Theorem \ref{Sensitivity 2} shows that the $\nabla V^i $ depends Lipshitz continuously on $\{\mu_.\}$.
Taking into account other assumptions we conclude that the unique point of maximum in
the expression
\[
\max_{u\in\U} \{  h_i(t,x,\mu_t,u)\nabla V^i(t,x,\{\mu_{\ge t}\})+J_i(t,x,\mu_t,u)\}
\]
has required properties.
\qed

\begin{remark}
Lipschitz continuity of the resulting control $\Gamma$ with respect to the parameter $\{\mu.\}$ was called the {\it feedback regularity property} in the pioneering paper \cite{HCM3}, where it was taken for granted as an additional assumption without verification. Our result above justifies this assumption of \cite{HCM3} by proving it in our much more general setting.
\end{remark}

Theorem \ref{thfeedbackHJB} gives
verifiable conditions for the assumption (i) of Theorem \ref{thkineticeq} to hold, which (together
with compactness criterion from Proposition \ref{propmomentcomp} feeding in Theorem \ref{thglobalexistgenkin})
 allows us to establish local well-posedness and global existence of solutions to \eqref{eqkineqmeanfieldnonM}
 in our basic setting of agents with a spatial and discrete component. This
completes tasks T1), T2) from the introduction. From now on we turn
to the $N$-particle approximations to the coupled mean-field evolution given by \eqref{eqkineqmeanfieldnonM}. For notational simplicity, the arguments and results in the following sections will be presented for one class agents, i.e. $K=1$. The extension to arbitrary $K$ is straightforward.


\section{Smooth dependence of nonlinear processes on initial data}
\label{secsmoothdepnonlinMark}

This section makes preparation for the analysis of the convergence of $N$ particle approximation in Section \ref{Convergence}.

Our aim is to prove smoothness of the solutions $\mu_t$ to kinetic equation \eqref{eq1propkineticeqnonlMark} with respect to initial data $\mu_0$. In fact we need the existence and regularity of the first and second order variational derivative of $\mu_t$ with resect to $\mu_0$. As a consequence, we are going to prove that the propagator of linear transformations on functionals of measures of the form
\begin{equation}
\label{eqproponkineq}
(\Phi^{s,t}(F))(\mu_s):=F (\mu_t(\mu_s)), \quad 0\leq s \le t
\end{equation}
(where $\mu_t(\mu_s)$ denotes the solution of the kinetic equation \eqref{eq1propkineticeqnonlMark}
with the initial condition $\mu_s$ at time $s$) preserves the set of twice differentiable functionals. This fact eventually would allow us to compare the propagators of $N$-particle processes with propagator \eqref{eqproponkineq}
of the limiting deterministic process via Proposition \ref{prop-propergatorProperty} with $\D$ being the space of functionals with sufficiently regular second order variational derivative.

We shall follow the line of arguments from \cite{Ko10}. However, essential simplification and, at the same time,
improvements will be achieved, in particular, by extending the results to time dependent family of generators and, most importantly, by weakening the requirements on the regularity of the coefficients of L\'evy-Khintchine type operators
\eqref{fellergenerator 3} with respect to position $z$, which is of principle importance for applications to MFG.

Let us start with an abstract kinetic equation \eqref{eq1propkineticeqnonlMark} in a dual Banach space $\B^*$
satisfying the assumptions of Theorem \ref{propkineticeqnonlMark}.
Given a family of initial data $\mu_0^{\al}$ depending on a real parameter $\al$, we are interested in the derivative
\begin{equation}
\label{eqdefderwithpar}
 \xi_t^\al=\frac{\pa \mu_t^{\al}}{\pa \al},\quad \forall t\geq 0
\end{equation}
of the corresponding solutions $\mu_t^{\al}$ with respect to this parameter.

Differentiating \eqref{eq1propkineticeqnonlMark} (at least formally for the moment)
with respect to $\al$ yields the equation
\begin{equation}
\label{eqforderpar}
 \frac{d}{dt}(f,\xi_t^\al)
 =(A[t,\mu_t^{\al}]f, \xi_t ^\al)+(\mathcal{D}_{\xi_t^\al}A[t,\mu_t^{\al}]f,\mu_t^{\al})
\end{equation}
with the initial condition
\begin{equation}
\label{eqdefderwithparinit}
\xi_0=\xi_0^\al=\frac{\pa \mu_0^{\al}}{\pa \al},
\end{equation}
where
\begin{equation}
\label{eqdeffirstGatderwithinitial}
 \mathcal{D}_{\xi}A[t,\mu]
 = \lim_{s \rightarrow 0_+} \frac{1}{s}(A[t,\mu+s\xi] - A[t,\mu]),\quad \forall \xi\in \B^*
\end{equation}
denotes the Gateaux derivatives of $A[t,\mu]$, where it is assumed that the definition of $A[t,\mu]$ can be extended to a neighborhood of $\M$ in $\D^*$.

The crucial point for what follows is the observation that one cannot expect equation \eqref{eqforderpar} to be well-posed in $\B^*$ making usual results from the theory of vector-valued ODE unapplicable.
Thus we are aiming at solving it in $\D^*$, so that the derivatives with respect to initial data are going to live in a different space than the solution itself. Moreover, even for solving equation \eqref{eqforderpar} in $\D^*$
(as is done in \cite{Ko07}, \cite{Ko11} for time-homogeneous family $A[\mu]$), requires strong smoothness assumptions on the family $A[t,\mu]$, which are not appropriate for analyzing MFG. Hence our plan is to solve (\ref{eqforderpar}) in a mild form, and then show that the solution does represent the derivative \eqref{eqdefderwithpar}.

\begin{assumption}
The Gateaux derivative (\ref{eqdeffirstGatderwithinitial}) exists in the norm-topology $\D^*$ and is regular enough so that, given $\mu \in \B^*$, the bilinear
form  $(\mathcal{D}_{\xi}A[t,\mu]f,\mu)$ on the arguments $\xi \in \D^*$ and $f \in \D$, from  \eqref{eqforderpar} can be written in the form
 \begin{equation}
\label{eqGatbilincont}
 (\mathcal{D}_{\xi}A[t,\mu]f,\mu)= (\phi_t[\mu] f, \xi)
\end{equation}
with a bounded linear operator $\phi_t[\mu]: \D\to \D$.
\end{assumption}

We shall see later that this assumption is satisfied for the case where $A[t,\mu]$ is of L\'evy-Khintchine form \eqref{fellergenerator 3} with coefficients depending smoothly on
$\mu$ (but, what is important, not necessarily so on positions $z$).

 Thus weak forward equation \eqref{eqforderpar} rewrites as
\begin{equation}
\label{eq2forderpar}
 \frac{d}{dt}(f,\xi_t^\al)
 =(A[t,\mu_t^{\al}]f, \xi_t ^\al)+(\phi_t[\mu_t^{\al}] f, \xi_t^\al),
\end{equation}
which is dual to the backward equation in $\D$ of the form
\begin{equation}
\label{eq3forderpar}
 \frac{d}{dt}f_t^{\al}
 =-A[t,\mu_t^{\al}]f_t^{\al}-\phi_t[\mu_t^{\al}]f_t^{\al}.
\end{equation}
Let us stress again that the operator $\phi_t[\mu]$ acts from $\D\to \D$ (and is not defined on $\B$).

As in Section \ref{Cauchy HJB}, assuming the family $A[t,\mu_t^{\al}]$ generates a backward propagator
$U_\al^{t,r}$, $t\leq r$, in $\B$ with the common invariant domain $\D$, with the dual forward propagator $\tilde U^{r,t}_\al$, we can write the
mild forms of equations \eqref{eq3forderpar} with a terminal condition $f_r^\al$ and \eqref{eq2forderpar} with an initial condition $\xi_t^\al$ respectively as

\begin{equation}
\label{eq4forderpar}
f^{\al}_t = U_\al^{t,r} f_r^{\al} + \int_t^r U_\al^{t,s} \phi_s[\mu_s^{\al}]f_s^{\al} \,  ds, \quad t\le r,
\end{equation}
and
\begin{equation}
\label{eq5forderpar}
\xi_r^{\al} = \tilde U_\al^{r,t} \xi_t^{\al} + \int_t^r \tilde U_\al^{r,s} \phi_s^*[\mu_s^{\al}] \xi_s^{\al} \,  ds,
\quad t\le r,
\end{equation}
or in the weak form
\begin{equation}
\label{eq6forderpar}
(f, \xi_r^{\al}) = (U^{t,r}_{\al} f, \xi_t^{\al}) + \int_t^r (\phi_s[\mu_s^{\al}] U^{s,r}_{\al}f, \xi_s^{\al}) \,  ds, \quad f\in \D.
\end{equation}
where $\phi^*_t[\mu]$ is the dual operator of $\phi_t[\mu]$. As we mentioned, we are going to solve these integral equations not insisting on the fact that
their solutions are related in any way to the original equation \eqref{eq2forderpar}.

It is clear (and comes from the standard arguments of perturbation theory) that equations
 \eqref{eq4forderpar}
and \eqref{eq5forderpar} have unique solutions $\digamma^{t,r}_{\al} f_r^{\al}$ and
$\tilde \digamma^{r,t}_{\al} \xi_t^{\al}$, where the propagators $\digamma^{t,r}_{\al}$ and its dual
$\tilde \digamma^{r,t}_{\al}$ are given by the convergent (in $\D$ and $\D^*$ respectively) perturbation series
\begin{equation}
\label{eqperturbseriesforprop}
 \digamma^{t,r}_{\al}=U^{t,r}_{\al}
 +\sum_{m=1}^{\infty} \int_{t\le s_1\le \cdots \le s_m\le r}
 U^{t,s_1}_{\al}\phi_{s_1}U^{s_1,s_2}_{\al}  \cdots  \phi_{s_m}U^{s_m,r}_{\al}\, ds_1  \cdots  ds_m,
 \end{equation}
\begin{equation}
\label{eqperturbseriesforpropdu}
 \tilde \digamma^{r,t}_{\al}=\tilde U^{r,t}_{\al}
 +\sum_{m=1}^{\infty} \int_{t\le s_1\le \cdots \le s_m\le r}
 \tilde U^{r,s_m}_{\al}\phi_{s_m}\tilde U^{s_m,s_{m-1}}_{\al}  \cdots  \phi_{s_1}\tilde U^{s_1,t}_{\al}\, ds_1  \cdots  ds_m,
 \end{equation}
where $\phi_s=\phi_s[\mu_s^{\al}]$.

\begin{theorem}
\label{thnonlsensit}
Under the assumptions of Theorem \ref{propkineticeqnonlMark},
let
\[
 \xi_0^{\al}=\frac{\pa \mu_0^{\al}}{\pa \al}
 \]
 be defined in $\D^*$ for a certain family of initial data $\mu_0^{\al}$ (with the derivative existing in the norm topology of $\D^*$). Moreover, let $A[t,\mu]$ depends on $\mu$ smoothly in the sense that representation \eqref{eqGatbilincont} holds with a continuous in $t$ family
 of bounded linear operator $\phi_t[\mu]: \D\to \D$.
Then $\mu_t^{\al}$ is Lipshitz continuous in $\al$ in the topology $\D^*$, the r.h.s. of \eqref{eqdefderwithpar} exists for almost all $\al$ (the derivative taken in the norm-topology of $\D^*$) and equals almost everywhere to the unique solution $\xi_t^{\al}=\xi_t^{\al}(\xi_0^{\al}, \mu_0^{\al})$ to equation \eqref{eq5forderpar} given by \eqref{eqperturbseriesforpropdu}.
\end{theorem}

\begin{proof}
The main idea is to approximate $A[t,\mu]$ by bounded operators, use the standard sensitivity theory for
vector valued ODE and then obtain the required result by passing to the limit.

To carry our this program,
let us choose a family of bounded operators $A[t,\mu](n): \B \to \B$ and a family of bounded operators $\phi_t[\mu](n): \B \to \B$, $n=1,2,...$, that satisfy all the same conditions
as $A[t,\mu]$ and $\phi_t[\mu]$, and such that
$$\|(A[t,\mu](n)-A[t,\mu])g\|_\B \to 0\quad \text{and}\quad \|(\phi_t[\mu](n)-\phi_t[\mu])g\|_\D \to 0$$
for all $g\in \D$. We shall use the same notations for corresponding solutions adding dependence on $n$ for all objects
constructed from $A[t,\mu](n)$.

As such approximations for $A[t,\mu]$, one can use either standard Iosida approximation
(which is convenient in abstract setting) or, in case of the generators of Feller processes, the generators of approximating pure-jump Markov processes.

By the standard theory of ODE in Banach spaces (see e.g. \cite{Mart} or Appendices D,F in \cite{Ko10}), the
equation for $\mu_t^\al(n)$ and $\xi_t^\al(n)$ are both well posed in the strong
sense in both $\B^*$ and $\D^*$, and
 $\xi_t^\al(n)$ represent the derivatives of $\mu_t^{\al}(n)$ in both $\B^*$ and $\D^*$.
  Consequently
 \begin{equation}
\label{eq1thnonlsensit}
 \mu_t^{\al}(n)-\mu_t^{\al_0}(n)= \int_{\al_0}^{\al} \xi_t^{\be}(n)\, d\beta\quad \forall \al,\al_0\in\R
\end{equation}
holds as an equation in $\D^*$ (and in
$\B^*$ whenever $\xi_0^{\al}\in \B^*$).

By the perturbation theory applied to nonlinear kinetic equations (see
\cite{Ko10}, Section 2.1 or \cite{Ko11}), $\mu_t^{\al}(n)$ converge to $\mu^{\al}_t$ in the norm-topology of
$\D^*$. Moreover, from the perturbation series representation \eqref{eqperturbseriesforpropdu}, we
deduce the convergence of $\xi_t^{\al}(n)$ to $\xi^{\al}_t$ in the norm-topology of
$\D^*$. Hence, we
can pass to the limit $n\to \infty$ in equation \eqref{eq1thnonlsensit} in the
norm topology of $\D^*$ yielding the equation
\begin{equation}
\label{eq2thnonlsensit}
 \mu_t^{\al}-\mu_t^{\al_0}= \int_{\al_0}^{\al} \xi_t^{\be} \, d\beta.
\end{equation}
This implies the statement of the Theorem.
\end{proof}

\begin{remark}
\label{remnonsmoothxi}
To establish continuity of $\xi^{\al}$ in $\al$, one seems to require much stronger regularity
 assumptions (see Remark \ref{remnonsmoothxiagain} for these assumptions).
And without this continuity one can get the derivative of the r.h.s. of \eqref{eq2thnonlsensit}
only almost everywhere. Fortunately, for variational derivatives that we are aiming at, we are able to circumvent this difficulty.
\end{remark}

Let us turn to the L\'evy-Khintchine type generators
\begin{equation}
\label{fellergenerator 31}
\begin{split}
&A[t,\mu]f(z)=\frac{1}{2}(G(t,z,\mu)\nabla,\nabla)f(z)+ (b(t,z,\mu),\nabla f(z))\\
 &+\int (f(z+y)-f(z)-(\nabla f (z), y){\bf 1}_{B_1}(y))\nu (t,z,\mu,dy),
\end{split}
\end{equation}
in $\B=C_{\infty}(\R^d)$
such that each of them generates a Feller process with one and the same invariant core $\D= C^2_{\infty}(\R^d)$. In order to do differentiation of the operators $A[t,\mu]$ with respect to $\mu$,  the operators $A[t,\mu]$ should be well defined on a neighbourhood of $\P(\R^d)$. So we introduce the set
\[
\M_{(\lambda_1,\lambda_2)}: = \{\M_\lambda: =\lambda \P(\R^d);  \lambda_1<\lambda<\lambda_2\}
\]
with $0<\lambda_1<1<\lambda_2$.

We also need the following notations. Let $C^{2\times2}_{\infty}(\R^d)$ denote the subspace of functions $F(x)$ from $C^2_\infty(\R^{d})$ such that the elements of 2nd derivative $\partial^2 F/\partial x^2$ are defined and
belong to $C^2_{\infty}(\R^{d})$ as functions of the second variable, equipped with the norm
\[
\|F\|_{C^{2\times2}_{\infty}(\R^{d})}:=\|F\|_{C^2_\infty(\R^{d})}+\sup_{x,y,i,j,l,k}\left\|\frac{\partial ^4 F(x,y)}{\partial x_i\partial x_j\partial y_k\partial y_l}\right\|.
\]

Let $C^{1,2}_{\infty}(\M_{(\lambda_1,\lambda_2)})$ denote the subspace of functionals $F(\mu)$ from
$C(\M_{(\lambda_1,\lambda_2)})$ such that the 1st variational derivative $\de F/\de \mu (x)$ is defined and
belongs to $C^2_{\infty}(\R^d)$ as a function of $x$ uniformly in $\mu$. This subspace becomes a Banach space when equipped with the norm
\[
\|F\|_{C^{1,2}_{\infty}(\M_{(\lambda_1,\lambda_2)})}
:=\sup_{\|\mu\|\leq 1} \left \|\frac{\de F}{\de \mu (.)}\right\|_{C_{\infty}^2(\R^d)}.
\]
Similarly we define the Banach space $C^{2,2}_{\infty}(\M_{(\lambda_1,\lambda_2)})$ of functionals with the norm
\[
\|F\|_{C^{2,2}_{\infty}(\M_{(\lambda_1,\lambda_2)})}
:=\sup_{\|\mu\|\leq 1}\left \|\frac{\de^2 F}{\de \mu (.) \de \mu(.)}\right\|_{C_{\infty}^2(\R^{d})},
\]
and $C^{2,2\times 2}_{\infty}(\M_{(\lambda_1,\lambda_2)})$ of functionals with the norm
\[
\|F\|_{C^{2,2\times 2}_{\infty}(\M_{(\lambda_1,\lambda_2)})}
:=\sup_{\|\mu\|\leq 1}\left \|\frac{\de^2 F}{\de \mu (.) \de \mu(.)}\right\|_{C_{\infty}^{2 \times 2}(\R^{d})}.
\]
It turns out (see Theorem \ref{thnonlsensitad3} below) that the space $C^{2,2\times 2}_{\infty}(\M_{(\lambda_1,\lambda_2)})$ (and not a more natural space $C^{2,2}_{\infty}(\M_{(\lambda_1,\lambda_2)})$) forms an invariant core for propagator \eqref{eqproponkineq}.

The Gateaux derivatives \eqref{eqdeffirstGatderwithinitial} are often specified via variational derivatives
\begin{equation}
\label{eqdeffirstGatvar}
 \frac{\de A[t,\mu]}{\de \mu (x)}=\mathcal{D}_{\de_x}A[t,\mu]
 = \lim_{s \rightarrow 0_+} \frac{1}{s}(A[t,\mu+s\de_x] - A[t,\mu]).
\end{equation}
By formal differentiation, as above, we expect the variational derivatives
\begin{equation}
\label{firstderivative}
\xi_t(x,\mu_0)=\frac{\de \mu_t}{\de \mu_0 (x)}
=\lim_{s \rightarrow 0_+} \frac{1}{s}(\mu_t(\mu_0+s\de_x) - \mu_t(\mu_0)),
\end{equation}
if they exist, to satisfy the mild version of equation
\begin{equation}
\label{eqforfirstdergenkineq}
 \frac{d}{dt}(f,\xi_t(x,\mu_0))
 =\int_{\R^d} \left(A[t,\mu_t]f(v)
 +\int_{\R^d} \frac{\delta A[t,\mu_t]}{\delta \mu_t (v)}f(z) \mu_t(dz)\right)
\xi_t(x,\mu_0,dv),
\end{equation}
where
\begin{equation}
\label{fellergeneratorvardif}
\begin{split}
&\frac{\delta A[t,\mu]}{\delta \mu (v)}f(z)
=\frac{1}{2}\left(\frac{\delta G(t,z,\mu)}{\delta \mu (v)}\nabla,\nabla\right)f(z)
+ \left(\frac{\delta b(t,z,\mu)}{\delta \mu (v)},\nabla f(z)\right)\\
 &+\int_{\R^d} (f(z+y)-f(z)-(\nabla f (z), y){\bf 1}_{B_1}(y))\frac{\delta \nu(t,z,\mu, dy)}{\delta \mu (v)}.
\end{split}
\end{equation}
When there is no ambiguity, we sometimes write $\xi_t(x)=\xi_t(x,\mu_0)$.
Equation (\ref{eqforfirstdergenkineq}) naturally rewrites in form \eqref{eqGatbilincont} with
\begin{equation}
\label{fellergeneratorvardif1}
(\phi_t[\mu]f)(v)=\int_{\R^d} \frac{\delta A[t,\mu]}{\delta \mu (v)}f(z) \mu(dz).
\end{equation}

\begin{theorem}
\label{thnonlsensitad}
Let $\B=C_{\infty}(\R^d)$,
$\D=C^2_{\infty}(\R^d)$, $\M=\M_{\la}$ for each $\la \in (\lambda_1,\lambda_2)$, $A[t,\mu]$ be of the L\'evy-Khintchine form \eqref{fellergenerator 31}. Suppose\\
(i) for any $\{\mu_.\}\in C_{\mu}([0,T],\M_\lambda(\D^*))$, the operator curve $A[t,\{\mu_.\}]:\D\mapsto \B$ generate a strongly continuous backward propagator of bounded linear operators $U^{t,s}[\{\mu_.\}]$ in $\B$, $0 \leq t \leq s$, on the common invariant domain $\D$, such that
	\begin{equation}
		||U^{t,s}[\{\mu.\}]||_{\D\mapsto \D}\leq c_1\,\, \text{and}\,\,||U^{t,s}[\{\mu.\}]||_{\B\mapsto \B}\leq c_2, \quad t\leq s,
	\end{equation}
for some positive constants $c_1,c_2$, and with their dual propagators $\tilde{U}^{s,t}[\{\mu_.\}]$ preserving the set $\M_\lambda$;\\
(ii) the coefficients
\begin{equation}
\label{eq1thnonlsensitad1}
G(t,z,\cdot),\, b(t,z,\cdot),\, \nu(t,z,\cdot) \in C^{1, 2}_{\infty}(\M_{(\lambda_1,\lambda_2)})
\end{equation}
uniformly in $t,z$ as functions of $\mu$, with uniformly bounded derivatives. \\

Then
the variational derivatives
$\xi_t(x)$ of the solutions $\mu_t=\mu_t(\mu_0)$ with respect to $\mu_0$ are well defined as elements of $\D^*$
and are twice continuously differentiable in $x$ (again in the topology of $\D^*$) with uniformly bounded
derivatives. Moreover, they represent unique solutions of the mild forms of equation \eqref{eqforfirstdergenkineq}
with the initial conditions $\xi_0(x)=\de_x$.
\end{theorem}

\begin{proof}
By the arguments given before Theorem \ref{thnonlsensit}
the solutions $\xi_t(x,\mu_0)$ to the mild form of equation
\eqref{eqforfirstdergenkineq} are uniquely defined in $\D^*$.
Moreover, as $\xi_t(x)$ solve linear equations, their derivatives with respect to initial data $\xi_0(x)(=\de_x)$ solve the same linear equations. Therefore, since
\[
\frac{\pa \de_x}{\pa x}\in (C^1_{\infty}(\R^d))^*\quad\text{and}\quad \quad \frac{\pa^2 \de_x}{\pa x^2}\in \D^*= (C^2_{\infty}(\R^d))^*,
\]
the derivatives
\[
\frac{\pa \xi_t (x)}{\pa x}, \quad \frac{\pa^2 \xi_t (x)}{\pa x^2}
\]
are well defined in $\D^*$ and solve the same equation as $\xi_t(x,\mu_0)$ do.

It remains to establish that $\xi_t(x)$ do represent variational derivatives.
To this end, let us follow the same argument as in Theorem \ref{thnonlsensit}. Approximating $A[t,\mu]$ by bounded operators $A[t,\mu](n)$ and passing to the limit in the integral equation defining variational derivatives
 \[
 F(Y+\delta_x)-F(Y) =\int_0^1 \frac{\de F(Y+s \delta_x)}{\de Y(x)}\, ds
 \]
(see e.g. (F.1) in \cite{Ko10}) leads to the equation
\[
\mu_t (\mu_0+\de_x)-\mu_t (\mu_0)=\int_0^1 \xi_t(x,\mu_0+s\de_x) \, ds.
\]
Unlike equation \eqref{eq2thnonlsensit} above, the functions on the r.h.s. are continuous, which allows us to
deduce the required property of $\xi_t(x)$.
\end{proof}

\begin{remark}
\label{remnonsmoothxiagain}
As $\xi_t(x)$ play a key role in our exposition, some further details seem to be in order.
Firstly, if the propagator generated by $A[t,\mu]$ has smoothing property \eqref{smooth property}
and moreover, $A[t,\mu]$ depends on $\mu$ only via the drift coefficients (that is $L_i$ from
\eqref{fellergeneratorwithdriftcont} does not depend on $\mu$, as is the case in the pioneering papers \cite{HCM3} and \cite{LL2007}), then the propagator resolving equation \eqref{eq4forderpar} (or more precisely, the mild form of the equation dual to \eqref{eqforfirstdergenkineq}) acts by bounded operators in $\B=C_{\infty}(\R^d)$. Hence in this case the variational derivatives $\xi_t(x)$ belong to the same space $\B^*=\mathbf{M}^{sign}(\R^d)$ as the evolution $\mu_t$ itself making all continuity arguments much more transparent. Moreover, in the most of examples one can even show that
the measures $\xi_t(x)$ have density with respect to Lebesgue measure, i.e. they are represented by functions from $L^1(\R^d)$. This situation was described in detail in \cite{Ko07}. Secondly, only assuming that
the propagator generated by $A[t,\mu]$ has smoothing property \eqref{smooth property} but also extended to smooth functions, i.e. it also satisfies
\begin{equation}
\label{smooth propertymore}
\|{U}^{t,s}\phi\|_{C_{\infty}^2(\mathbf{R}^d)} \leq w(s-t)\|\phi\|_{C^1_{\infty}(\mathbf{R}^d)},
\end{equation}
then, as one easily sees, the propagator resolving the mild form of the equation dual to \eqref{eqforfirstdergenkineq} acts by bounded operators in $C^1_{\infty}(\R^d)$. Hence $\xi_t(x)$ belong to the space $(C^1_{\infty}(\R^d))^*$ and continuity in this space of the integrand on r.h.s. of \eqref{eq2thnonlsensit} is easily established.
\end{remark}

Let us analyze similarly the second variational derivative
\begin{equation}
\begin{split}
\eta_t(x,y)=&\eta_t(x,y,\mu_0)=
\frac{\de^2 \mu_t}{\de \mu_0(x) \de \mu_0(y)}\\
=&\lim_{s \rightarrow 0_+} \frac1s \left(\frac{\de \mu_t}{\de \mu_0 (x)}(\mu_0+s\de_y)
-\frac{\de \mu_t}{\de \mu_0 (x)}(\mu_0)\right).
\end{split}
\end{equation}
Differentiating (again formally in the first place) equation \eqref{eqforfirstdergenkineq}, yields the equation

 \begin{equation}\label{eqforsecdergenkineq}
 \begin{split}
\frac{d}{dt}(f,\eta_t(x,y))& =\int_{\R^d} \left(A[t,\mu]f(v)
 +\int_{\R^d} \frac{\delta A[t,\mu]}{\delta \mu_t (v)}f(z) \mu_t(dz)\right)
\eta_t(x,y,dv)\\
 &+ \int_{\R^{2d}} \Big( \frac{\delta A[t,\mu]}{\delta \mu_t (u)}f(v)
  + \frac{\delta A[t,\mu]}{\delta \mu_t (v)}f(u)\\
  & \hspace{3em}+ \int_{\R^d} \frac{\delta ^2A[t,\mu]}{\delta \mu_t (v) \delta \mu_t(u)}f(z)
  \mu_t(dz)\Big) \xi_t(x,dv) \xi_t(y,du)
  \end{split}
  \end{equation}
with $\eta_0(x,y)=0$. Following the similar argument to the first order derivative, we shall look at the solution to (\ref{eqforsecdergenkineq}) only in the mild version:
\begin{equation}
\label{eqforsecdergenkineqmild}
\eta_t(x,y) = \tilde U^{t,0} \eta_0(x,y) + \int_0^t \tilde U^{t,s} (\phi_s^*[\mu_s] \eta_s(x,y) +G_s(x,y))\,  ds,
\end{equation}
where $\phi_s^*[\mu]$ is the dual operator of $\phi_s[\mu]$ given in \eqref{fellergeneratorvardif1} and $G_s\in \D^*$ given by
\begin{equation}
\begin{split}
\label{G}
(g,G_s(x,y))=
\int_{\R^{2d}}\Big( &\frac{\delta A[s,\mu]}{\delta \mu_s (u)}g(v)
  +\frac{\delta A[s,\mu]}{\delta \mu_s (v)}g(u)\\
  +& \int _{\R^d}\frac{\delta ^2A[s,\mu]}{\delta \mu_s (v) \delta \mu_s(u)}g(z)
  \mu_s(dz)\Big) \xi_s(x,dv) \xi_s(y,du).
\end{split}
\end{equation}

For equation (\ref{eqforsecdergenkineqmild}) to make sense, we need only to make sure that the last term in (\ref{G}) is well defined for
$\xi_s \in \D^*=(C^2_{\infty}(\R^d))^*$. This leads to the following extension of Theorem \ref{thnonlsensitad} to the second variational derivative.

\begin{theorem}
\label{thnonlsensitad2}
Under the assumptions of Theorem \ref{thnonlsensitad},
suppose additional regularity on the coefficients, namely
\begin{equation}
\label{eq1thnonlsensitad}
G(t,z,\cdot),\, b(t,z,\cdot),\,  \nu(t,z,\cdot) \in C^{2,2\times 2}_{\infty}(\M_{(\lambda_1,\lambda_2)})
\end{equation}
uniformly in $t,x$ as functions of $\mu$, with uniformly bounded derivatives. Then the second variational derivatives
$\eta_t(x,y)$ of the solutions $\mu_t=\mu_t(\mu_0)$ with respect to $\mu_0$ are well defined as elements of $\D^*$
and are twice continuously differentiable as functions of $x$ and $y$ (again in the topology of $\D^*$) with uniformly bounded derivatives. Moreover, $\eta_t(x,y)$ represent unique solutions of the mild forms of equation \eqref{eqforsecdergenkineq} with $\eta_0(x,y)=0$.
\end{theorem}

\begin{proof}
It is almost straightforward extension of Theorem \ref{thnonlsensitad}. However, differentiating equation \eqref{eqforsecdergenkineq} with respect to $x$ and $y$, one has to take into account that dependence on $x,y$ is not only via initial conditions, but also via the r.h.s. containing  $\xi_t(x,dv) \xi_t(y,du)$. This product form allows to differentiate
twice in $x$ and additionally twice in $y$, as need to get a function of class $C^{2,2\times 2}_{\infty}(\M_{(\lambda_1,\lambda_2)})$.
\end{proof}

\begin{remark}
If $G,b,\nu$ depend on $\mu$ as smooth functions of some multiple integrals (which is usually the case in application)
of the form
\[
\int g(x_1,\cdots, x_n) \mu(dx_1)\cdots \mu(dx_n)
\]
 with a smooth function $g$, then all variational derivatives
are defined and are smooth. Hence smoothness of functions \eqref{eq1thnonlsensitad1} and \eqref{eq1thnonlsensitad} are very natural and not much restrictive assumptions.
\end{remark}

\begin{remark}
In case $A[t,\mu]=A[t]$ not depending on $\mu$, i.e. in case of a usual Markov process,
$\xi_t(x,\mu_0)$ does not depend on $\mu_0$ and coincides with the transition probability
$P(t,x,dy)$ of this process. Moreover, $\eta_t(x,y)$ vanishes.
\end{remark}

Finally we are able to analyze
the transformations generated by the flow $\mu_0=\mu \mapsto \mu_t$ on functionals, that is the
transformations $(\Phi^t(F))(\mu)=F (\mu_t(\mu))$, or the corresponding backward propagators \eqref{eqproponkineq}.

These transformations are well defined as linear contractions on the space $C(\M_{(\lambda_1,\lambda_2)})$
of continuous in $\D^*$ functions on $\M_{(\lambda_1,\lambda_2)}$, since
\begin{equation}
\begin{split}
\| \Phi^{s,t}(F) \|_{C(\M_{(\lambda_1,\lambda_2)})}=&\sup_{\mu\in \M_{(\lambda_1,\lambda_2)}}|\Phi^{s,t}(F)(\mu) |\\
=\sup_{\mu\in \M_{(\lambda_1,\lambda_2)}}|F (\mu_t(\mu_s)) | &\leq \sup_{\mu\in \M_{(\lambda_1,\lambda_2)}}|F (\mu) |\\
=&\|F\|_{C(\M_{(\lambda_1,\lambda_2)})}.
\end{split}
\end{equation}

\begin{theorem}
\label{thnonlsensitad3}
(i) Under the assumptions of Theorem \ref{thnonlsensitad}, the transformations
\eqref{eqproponkineq} form a backward propagator of uniformly bounded (for $s,t$ from any compact interval)
linear operators in $C^{1,2}_{\infty}(\M_{(\lambda_1,\lambda_2)})$.
(ii) Under the assumptions of Theorem \ref{thnonlsensitad2}, the transformations
\eqref{eqproponkineq} form a backward propagator of uniformly bounded (for $s,t$ from any compact interval)
linear operators in $C^{2,2\times 2}_{\infty}(\M_{(\lambda_1,\lambda_2)})$.
\end{theorem}

\begin{proof}
It follows from Theorems \ref{thnonlsensitad} and \ref{thnonlsensitad2} and the formulas
\[
\frac{(\de \Phi^{0,t}(F))(\mu) }{\de \mu_0(x)}=\frac{\de F (\mu_t)}{\de \mu_0(x)}=\int_{\R^d} \frac{\de F (\mu_t)}{\de \mu_t(z)}\xi_t(x,\mu_0,dz),
\]
\begin{equation}
\begin{split}
\frac{(\de^2 \Phi^{0,t}(F))(\mu) }{\de \mu_0(x)\de \mu_0(y)}=&\frac{\de^2 F (\mu_t)}{\de \mu_0(x)\de \mu_0 (y)}\\
=&\int _{\R^d}\frac{\de F (\mu_t)}{\de \mu_t(z)}\eta_t(x,y,\mu_0,dz)\\
&+\int_{\R^{2d}} \frac{\de^2 F (\mu_t)}{\de \mu_t(z) \de \mu_t(w)}\xi_t(x,\mu_0,dz)\xi_t(y,\mu_0,dw).
\end{split}
\end{equation}
\end{proof}

\section{Convergence of $N$-particle approximations (dynamic law of large numbers)}
\label{Convergence}

Here we start with the refining of some results on the weak LLN limit for interacting particles from Chapter 9 of \cite{Ko10} based on the improvements of the regularity results achieved in the previous section. Then we extend them in two directions: we prove that all the rates of convergence remain valid (i) if one of the agent would be allowed to deviate from the common behavior and (ii) if one uses integral functionals on measure paths (interpreted as integral payoffs in the game setting),
rather than  just on marginal measures.

To this end, let us return to operator $\widehat A_t^N$ of form \eqref{eq0Nparticlegen} from
the introduction (where we omit, for the moment, the dependence on $u$ and $\gamma$):
 \begin{equation}
 \label{eq1Nparticlegen}
\widehat A_t^N F(\de_{\x}/N)=\sum_{i=1}^N A^i[t,\mu]f (x_1,\cdots , x_N),
\end{equation}
where $f(\x)=F(\de_{\x}/N)$, $\x=(x_1,\cdots, x_N)$ and
\[
\mu= \frac1N \sum_{i=1}^N \de_{x_i}=\frac1N \de_{\x}.
\]
Assume now that $A[t,\mu]$ is of the L\'evy-Khintchine form \eqref{fellergenerator 31} in $\B=C_{\infty}(\R^d)$
and such that each $\widehat A_t^N$ generates a Feller process with invariant core $\D= C^2_{\infty}(\R^{dN})$.
The following calculations are performed in Chapter 9 of \cite{Ko10}.

\begin{prop}
\label{propexplicitappropgen}
 If $F$ is smooth enough, which means that all the
derivatives in the formulas below are well defined, then
\[
\widehat A_t^N F(\mu)= \int_{\R^d}\left[\left(A[t,\mu] \frac{\delta F}{\delta \mu
(.)}\right)(x)+\frac{1}{2N} \left(G(t,x, \mu)\frac{\pa}{\pa x},
\frac{\pa}{\pa y}\right)\frac{\delta^2 F}{\delta \mu(x)\delta
\mu(y)}\Big|_{y=x}\right]\mu(dx)
\]
\begin{equation}
\label{eqsecmanipulgen2}
 +\frac1N\int_0^1(1-s) \, ds \,
\int_{\R^{2d}}\left(\frac{\delta^2 F}{\delta \mu(.)\delta \mu(.)}
(\mu+\frac{s}{N}(\delta_{x+y}-\delta_x)), (\delta_{x+y}-\delta_x)^{\otimes
2}\right)\nu (t,x, \mu,dy)\mu(dx),
\end{equation}
where $\mu=\de_{\x}/N$.
 \end{prop}

 This formula yields a precise estimate for the deviation of operator \eqref{eq1Nparticlegen} from the
 limiting operator
\begin{equation}
\begin{split}
\label{Lambda}
 \Lambda_t F(\mu)
 &=\int_{\R^d} A[t,\mu]\left(\frac{\de F}{\de \mu (.)}\right)(z)
 \mu (dz)\\
 &=\left(A[t,\mu ]\left(\frac{\de F}{\de \mu (.)}\right),\mu\right).
\end{split}
\end{equation}

\begin{remark}
If the functionals $F$ are linear, i.e. of the form \eqref{linearF}: $F_t(\mu)=(f,\mu)$
for a $f\in C(\R^d)$, then the r.h.s of \eqref{Lambda} is the one of \eqref{eq1propkineticeqnonlMark}
(with $\mu_t=\mu$).
\end{remark}

According to Theorem \ref{thnonlsensitad3}, this family of operators generates propagator \eqref{eqproponkineq}
on the common invariant domains $C^{1, 2}_{\infty}(\M_{(\lambda_1,\lambda_2)})$ or $C^{2,2\times 2}_{\infty}(\M_{(\lambda_1,\lambda_2)})$.
Thus we need $F\in C^{1, 2}_{\infty}(\M_{(\lambda_1,\lambda_2)})$ to have $\Lambda_t F$ to be well defined and
$F\in C^{2,2\times 2}_{\infty}(\M_{(\lambda_1,\lambda_2)})$ to have an estimate of order $1/N$ for the deviation of $\Lambda_t F$
from $\widehat A_t^NF$.

Let $U^{s,t}_N$, $s\leq t$, denote the backward propagator on $C(\PC^N_{\delta}(\R^d))$ generated by \eqref{eq1Nparticlegen} and $\Phi^{s,t}$ is defined in \eqref{eqproponkineq}. From Proposition \ref{prop-propergatorProperty}, we have:
\begin{equation}
\label{U}
[(U^{s,t}_N- \Phi^{s,t})F] (\mu)
= \int_s^t [U^{s,r}_N(\widehat A_r^N-\Lambda_r)\Phi^{r,t} F] (\mu) \, ds,
\end{equation}
for
$F \in C^{2,2\times 2}_{\infty}(\M_{(\lambda_1,\lambda_2)})$ and $\mu \in \PC^N_{\delta}(\R^d)$.
By Theorem \ref{thnonlsensitad3} and Proposition \ref{propexplicitappropgen}, equation (\ref{U}) implies
\begin{equation}
\label{eqsecmanipulgen2a}
\sup_{\mu \in \PC^N_{\delta}(\R^d)} |[(U^{s,t}_N- \Phi^{s,t})F] (\mu)|
\le \frac{C(T)}{N} (t-s) \|F \|_{C^{2,2\times 2}_{\infty}(\M_{(\lambda_1,\lambda_2)})}
\end{equation}
for $0\le s\le t \le T$, with some constant $C(T)$ arising from the bounds on the propagators $\Phi^{s,t}$
on $C^{2,2\times 2}_{\infty}(\M_{(\lambda_1,\lambda_2)})$ from
Theorem \ref{thnonlsensitad3}.

This leads to the following result.
\begin{theorem}
\label{thLLNgenrate}
Under the assumptions of Theorem \ref{thnonlsensitad2},
suppose the initial conditions
\[
\mu_0^N=\frac1N (\de_{X_{1,0}^N}+\cdots +\de_{X_{N,0}^N})
\]
converge in $\D^*=(C^2_{\infty}(\R^d))^*$, as $N\to \infty$, to a probability law $\mu_0 \in \PC(\R^d)$ so that
\begin{equation}\label{estimatemu}
\|\mu_0^N -\mu_0\|_{\D^*}\le \frac{\kappa_1}{N}.
\end{equation}
with a constant $\kappa_1>0$. Then, for $t\in [0,T]$ with any $T\geq 0$,
\begin{equation}
\label{th1LLNgenrate}
|[U^{0,t}_N F](\mu_0^N)-[\Phi^{0,t}(F)](\mu_0)|
\le \frac{ C(T)}{N} (t\|F \|_{C^{2,2\times2}_{\infty}(\M_{(\lambda_1,\lambda_2)})}+\kappa_1)
\end{equation}
with a constant $C(T)$.
\end{theorem}

\begin{proof}
For
$F \in C^{2,2\times 2}_{\infty}(\M_{(\lambda_1,\lambda_2)})$, we have
 \begin{equation*}
 \begin{split}
&\left|[U^{0,t}_N (F)](\mu_0^N)-[\Phi^{0,t}(F)](\mu_0)\right|\\
\leq &\left|[U^{0,t}_N (F)-\Phi^{0,t}(F)](\mu_0^N)\right|+\left|[(\Phi^{0,t}(F))(\mu_0^N)-(\Phi^{0,t}(F))(\mu_0)]\right|.
\end{split}
\end{equation*}
Estimating the first term by \eqref{eqsecmanipulgen2a} and the second one by the Lipshitz continuity
of the solutions to kinetic equation obtained in Theorem \ref{propkineticeqnonlMark} and assumption (\ref{estimatemu}), yields
\eqref{th1LLNgenrate}.
\end{proof}

Let us turn now to full operator \eqref{eq0Nparticlegen} with $u$ and $\ga$ included. Assume
that all agents follow some fixed strategy $\ga(t,x)$, but for one player, say the 1st one, who applies
 a different control $u_{1,t}=\tilde \ga (t,x)$. Then, instead of the operator \eqref{eq1Nparticlegen}, the controlled process of $N$ interacting agents will be generated by the operator
 \begin{equation}
 \begin{split}
 \label{eq2Nparticlegen}
&\widehat{A}_t^N[\ga, \tilde \ga] F(\de_{\x}/N)
\\
=&\widehat{A}_t^N[\ga]  F(\de_{\x}/N)
+[A^1(t,\mu, \tilde \ga (t,.))-A^1(t,\mu, \ga (t,.))]F(\de_{\x}/N),
\end{split}
\end{equation}
where $\mu=\de_{\x}/N$, $F\in C^{2,2\times 2}_{\infty}(\M_{(\lambda_1,\lambda_2)})$ and
\[
\widehat{A}_t^N[\ga]  F(\de_{\x}/N)
=\sum_{i=1}^N A^i[t,\mu, \ga]f (x_1,\cdots , x_N).
\]

\begin{remark}
If the agents only control their drifts, as we usually assume here, i.e. generators $A^1[t,\mu,u]$ has the form of \eqref{Ai}, then
\begin{equation}
\begin{split}
&[A^1(t,\mu, \tilde \ga (t,.))-A^1(t,\mu, \ga (t,.))]F(\de_{\x}/N)\\
=&(h_1(t,\mu,\tilde \ga (t,.))-h_1(t,\mu, \ga (t,.)), \nabla F(\de_{\x}/N)).
\end{split}
\end{equation}
\end{remark}

It is clear that the result of the application of the last term
 of \eqref{eq2Nparticlegen} yields an expression of order $1/N$, allowing us to conclude that
\[
\widehat{A}_t^N[\ga, \tilde \ga] F(\de_{\x}/N)
=\widehat{A}_t^N[\ga]F(\de_{\x}/N)+O(1/N),
\]
where $O(1/N)$ is uniform for $F$ from bounded subsets
$C^{2,2\times 2}_{\infty}(\M_{(\lambda_1,\lambda_2)})$. Hence
$\widehat{A}_t^N[\ga, \tilde \ga] F$ and $\widehat{A}_t^N[\ga]F$ have the same limit
$\Lambda_t[\ga]$ defined in (\ref{Lambda}), and the same rate of
convergence of order $1/N$ holds for $F$ from bounded subsets
$C^{2,2\times 2}_{\infty}(\M_{(\lambda_1,\lambda_2)})$.
However, instead of just comparing the marginal distributions of Markov processes $\mu_t^N[\ga, \tilde \ga]$
specified by the propagators $U^{0,t}_N[\ga,\tilde \ga]$ (and generated
by the $\widehat{A}_t^N[\ga, \tilde \ga]$), as in Theorem \ref{thLLNgenrate}, we also need to be able to compare their integral characteristics
(interpreted as payoffs in a game theoretic setting).

The following result is a simplified preliminary version of Theorem \ref{thLLNgenratedevinttag} below.
\begin{theorem}
\label{thLLNgenratedevint}
Suppose the assumptions of Theorem \ref{thnonlsensitad2} hold for the family of operators
$A[t,\mu,\ga (t,.)]$ with
\begin{equation}
\label{fellergenerator 33}
\begin{split}
&A[t,\mu,u]f(z)=\frac{1}{2}(G(t,z,\mu,u)\nabla,\nabla)f(z)+ (b(t,z,\mu,u),\nabla f(z))\\
 &+\int (f(z+y)-f(z)-(\nabla f (z), y){\bf 1}_{B_1}(y))\nu (t,z,\mu,u,dy),
\end{split}
\end{equation}
and a class of functions $\ga: \R^+ \times \R^d \to \U$, that are continuous in the first variable and Lipschitz continuous in the second one.

Suppose the initial conditions
\[
\mu_0^N=\frac1N (\de_{X_{1,0}^N}+\cdots +\de_{X_{N,0}^N})
\]
converge in $\D^*=(C^2_{\infty}(\R^d))^*$, as $N\to \infty$, to a probability law $\mu_0 \in \PC(\R^d)$ so that
\begin{equation*}
\|\mu_0^N -\mu_0\|_{\D^*}\le \frac{\kappa_1}{N}
\end{equation*}
with a constant $\kappa_1>0$.

Let $U^{0,t}_N[\ga,\tilde \ga]$ denote the propagator generated by $\widehat{A}_t^N[\ga, \tilde \ga]$ and
$\Phi^{0,t}[\ga]$ the propagator generated by $\Lambda_t[\ga]$.

(i) Then, for $t\in [0,T]$ with any $T>0$,
\begin{equation}
\begin{split}
\label{eq1thLLNgenratedevint}
&\big|[U^{0,t}_N[\ga,\tilde \ga] F](\mu_0^N)-[\Phi^{0,t}[\ga](F)](\mu_0)\big|\\
&\le  \frac{C(T)}{N} \Big(t\|F \|_{C^{2,2\times 2}_{\infty}(\M_{(\lambda_1,\lambda_2)})}+\kappa_1\Big)
\end{split}
\end{equation}
with a constant $C(T)$ not dependent on $\tilde \ga$.

(ii) Moreover,
let $J(t,\mu)\in C([0,T], C^{2,2\times 2}_{\infty}(\M_{\lambda_1,\lambda_2}))$. Then
 \begin{equation}
 \begin{split}
\label{eq2thLLNgenratedevint}
&\left|\E \int_t^T J(s,\mu_s^N [\ga, \tilde \ga]) \, ds -\int_t^T J(s,\mu_s [\ga]) \, ds \right|\\
&\le \frac{C(T)}{N}\Big((T-t)\|J\|_{C([0,T], C^{2,2\times 2}_{\infty}(\M_{\lambda_1,\lambda_2})) }+\kappa_1\Big),
\end{split}
\end{equation}
where $\mu_t^N[\ga, \tilde \ga]$ is the Markov process
specified by the propagators $U^{0,t}_N[\ga,\tilde \ga]$ and $\mu_t [\ga]$ is the solution to kinetic
equation \eqref{eqkineqmeanfield} with the initial $\mu_0$.
\end{theorem}

\begin{proof}
(i) This follows from \eqref{eq2Nparticlegen} by the same argument as in Theorem \ref{thLLNgenrate}.
(ii) Let us represent the integral $\int_0^t J(s,\mu_s^N [\ga, \tilde \ga]) \, ds$ as the limit of
Riemannian sums. Then \eqref{eq2thLLNgenratedevint} is obtained by applying \eqref{eq1thLLNgenratedevint} term-by-term and passing to the limit.
\end{proof}

\section{Mean field limit as an $\epsilon$-equilibrium}
\label{secepequilibrium}

In this section, we complete our program by performing task T4), that is by showing
that a solution to equation \eqref{eqkineqmeanfieldnonM} provides a perfect $\ep$- equilibrium for
the corresponding approximating game of $N$ players, with $\ep$ of order $1/N$. For simplicity, we consider one class of agents. Extensions to multi-class agents, using Theorem \ref{thfeedbackHJB}, is straightforward.

Let us recall, that for $\ep >0$, a strategy profile $\Gamma$ in a game of $N$ players with payoffs
$V_i(\Gamma)$, $i=1,\cdots, N$, is a {\it $\ep$-equilibrium} (or $\ep$-Nash equilibrium) if, for each player $i$ and an eligible strategy $u_i$,
\[
V_i(\Gamma)\geq V_i(\Gamma_{-i}, u_i)-\epsilon,
\]
where $(\Gamma_{-i}, u_i)$ denotes the profile obtained from $\Gamma$ by substituting the strategy of player $i$ with $u_i$. A profile of dynamic strategies in a dynamic game on a  time interval $[0,T]$ is called a {\it perfect $\ep$- equilibrium}, if it is an $\ep$-equilibrium for any subgame started any time $t\in [0,T]$
(see a systematics presentation of this notion in \cite{KolMal1} or any textbook on game theory,
original papers on perfect $\ep$-equilibria being \cite{R1980} and \cite{FL1986}).

To complete our program, we need to extend the estimate \eqref{eq2thLLNgenratedevint} to the case of $J$ depending
on the position of a tagged agent and its own strategy. To this end, let us look at the process of pairs
$(X_{1,t}^N, \mu_t^N)$, that is on a chosen tagged agent and the overall mass.

Let $C^{2;2,2\times 2}_{\infty}(\R^{d}\times\M_{(\lambda_1,\lambda_2)})$ denote the subspace of functionals $F(x,\mu)$ from $C_{\infty}(\R^{d}\times\M_{(\lambda_1,\lambda_2)})$, the space of continuous bounded functionals on $\R^d\times \M_{(\lambda_1,\lambda_2)}$,  such that for each $\mu\in \M_{(\lambda_1,\lambda_2)}$, the functional $F(.,\mu)\in C^2_\infty(\R^d)$ and for each $x\in\R^d$, $F(x,.)\in C^{2,2\times 2}_{\infty}(\M_{(\lambda_1,\lambda_2)})$.

Let $C([0,T]\times \U, C^{2;2,2\times 2}_{\infty}(\R^{d}\times\M_{(\lambda_1,\lambda_2)}))$ denote the subspace of functionals $F(t,x,\mu, u)$ from $C_{\infty}([0,T]\times\R^{d}\times\M_{(\lambda_1,\lambda_2)}\times \U)$, the space of continuous bounded functionals on $[0,T]\times\R^{d}\times\M_{(\lambda_1,\lambda_2)}\times \U$,  such that for each $(t,\mu,u)$, the functional $F(t,.,\mu,u)\in C^2_\infty(\R^d)$ and for each $(t,x,.,u)$, $F(t,x,.,u)\in C^{2,2\times 2}_{\infty}(\M_{(\lambda_1,\lambda_2)})$.

The generators of the process of pairs $(X_{1,t}^N, \mu_t^N)$
are defined on the space $C^{2;2,2\times 2}_{\infty}(\R^{d}\times\M_{(\lambda_1,\lambda_2)})$ and take the form
  \begin{equation}
 \label{eqNparticleplusdev}
\widehat A_{t;tag}^N[\ga, \tilde \ga] F(x_1,\de_{\x}/N):
=\left(A^1[t, \de_{\x}/N, \tilde \ga]+\widehat{A}_t^N[\ga, \tilde \ga]\right)F(x_1,\de_{\x}/N),
\end{equation}
where the 1st (resp. 2nd) operator in the sum acts on the 1st (resp. 2nd) variable of $F$ ($A^1$ just means the operator $A$ acting on the variable $x_1$).
Consequently, by equation (\ref{eq2Nparticlegen}) and Proposition \ref{propexplicitappropgen},
   \begin{equation}
 \label{eqNparticleplusdevrate}
\widehat A_{t; tag}^N[\ga, \tilde \ga] F(x_1,\de_{\x}/N)
=\left(A^1[t, \de_{\x}/N, \tilde \ga]+\La_t[\ga]\right)F(x_1,\de_{\x}/N)+O(1/N).
\end{equation}

\begin{theorem}
\label{thLLNgenratedevinttag}
Under the assumption of Theorem \ref{thLLNgenratedevint},
let $ U^{0,t}_{N;tag}[\ga,\tilde \ga]$ denote the propagator generated by $\widehat A_{t; tag}^N[\ga, \tilde \ga]$ and
$ \Phi^{0,t}_{tag}[\ga]$ the propagator generated by the family
\[
A^1[t, \mu, \tilde \ga]+\La_t[\ga]
\]
where $\mu=\de_{\x}/N$. Suppose additionally $X_{1,0}^N\in\R^d$ converge, as $N\to \infty$, to a point $X_{1,0}\in\R^d$ such that
\begin{equation}\label{estimatex}
||  X^N_{1,0}-X_{1,0}||\leq \frac{\kappa_2}{N}
\end{equation}
with a constant $\kappa_2>0$.

(i) Then, for $t\in [0,T]$ with any $T>0$,
\begin{equation}
\begin{split}
\label{eq1thLLNgenratedevint2}
&\Big|[U^{0,t}_{N;tag}[\ga,\tilde \ga] F](X^N_{1,0},\mu_0^N)-[\Phi^{0,t}_{tag}[\ga](F)](X_{1,0},\mu_0)\Big|\\
&\le  \frac{C(T)}{N} \Big((t+k_2)\|F \|_{C^{2;2,2\times 2}_{\infty}(\R^{d}\times\M_{(\lambda_1,\lambda_2)})}+\kappa_1\Big)
\end{split}
\end{equation}
with a constant $C(T)$ not dependent on $\tilde \ga$.

(ii) Moreover, if $J(t,x,\mu,u)\in C([0,T]\times \U, C^{2;2,2\times 2}_{\infty}(\R^{d}\times\M_{(\lambda_1,\lambda_2)})) $, then
 \begin{equation}
 \begin{split}
\label{eq2thLLNgenratedevint2}
&\left|\E \int_t^T J(s,X_{1,s}^N,\mu_s^N [\ga, \tilde \ga], \tilde \ga (s,X_{1,s}^N)) \, ds
 -\E\int_t^T J(s, X_{1,s}, \mu_s [\ga], \ga (s,X_{1,s})) \, ds \right|\\
 &\le \frac{C}{N}\Big((T+\kappa_2)\|J \|_{C([0,T]\times \U, C^{2;2,2\times 2}_{\infty}(\R^{d}\times\M_{(\lambda_1,\lambda_2)}))}+\kappa_1\Big)
 \end{split}
\end{equation}
where the pair $(X_{1,s}^N, \mu_t^N[\ga, \tilde \ga])$ is the Markov process
specified by the propagators $U^{0,t}_{N;tag}[\ga,\tilde \ga]$ and $\mu_s [\ga]$ is the solution to kinetic
equation \eqref{eqkineqmeanfield} with the initial $\mu_0$ and the Markov process $X_{1,s}$ is generated by $A^1[t,\mu_t[\ga],\tilde \ga]$.
\end{theorem}

\begin{proof}
(i) For $F\in C^{2;2,2\times 2}_{\infty}(\R^{d},\M_{(\lambda_1,\lambda_2)})$, we have
\begin{equation}
 \begin{split}
&\Big|[U^{0,t}_{N;tag}[\gamma, \hat \gamma]F](X_{1,0}^N, \mu_0^N)-[\Phi^{0,t}_{tag}[\gamma]F](X_{1,0},\mu_0)\Big|\\
\leq & \Big| [U^{0,t}_{N;tag}[\gamma, \hat \gamma]F](X_{1,0}^N, \mu_0^N)-[\Phi^{0,t}_{tag}[\gamma]F](X_{1,0}^N, \mu_0^N) \Big|\\
&+\Big| [\Phi^{0,t}_{tag}[\gamma]F](X_{1,0}^N, \mu_0^N)-[\Phi^{0,t}_{tag}[\gamma]F](X_{1,0},\mu_0^N) \Big|\\
&+\Big|[\Phi^{0,t}_{tag}[\gamma]F](X_{1,0}, \mu_0^N)-[\Phi^{0,t}_{tag}[\gamma]F](X_{1,0},\mu_0) \Big|.
 \end{split}
\end{equation}
Estimating the first term by \eqref{eqsecmanipulgen2a}, the second one by assumption (\ref{estimatex}) and the third one by the assumption \eqref{estimatemu} and the Lipschitz continuity
of the solutions to kinetic equation obtained in Theorem \ref{propkineticeqnonlMark}, yields
\eqref{eq1thLLNgenratedevint2}.

(ii) Let us represent both integrals $\int_t^T J(s,X_{1,s}^N,\mu_s^N [\ga, \tilde \ga], \tilde \ga (s,X_{1,s}^N)) \, ds$ and $\int_t^T J(s, X_{1,s}, \mu_s [\ga], \tilde \ga (s,X_{1,s})) \, ds$
as the limits of
Riemannian sums. Then \eqref{eq2thLLNgenratedevint2} is obtained by applying \eqref{eq1thLLNgenratedevint2} term-by-term and passing to the limit.
\end{proof}

\begin{remark}
\label{remtaggedtangentprocess}
Theorem \ref{thLLNgenratedevinttag} remains valid if the initial points $X_{1,0}^N$ are random and their distributions
converge to a certain law $\nu$. If $\tilde \ga=\ga$, the limiting process $X_{1,t}$ is the tangent Markov process to the nonlinear process $\mu_t$, see Remark \ref{tangentprocess}.
\end{remark}

Theorem \ref{thLLNgenratedevinttag} and all results of Sections \ref{secsmoothdepnonlinMark} and \ref{Convergence}
had nothing to do with optimization (and with the theory of HJB equation of Section \ref{Sensitivity}). They belong
to statistics of interacting particles and are of interest in their own right. The results hold for all examples given in Section \ref{examples} (see Theorem \ref{propmainmultiagentexamples}).

Recall that we are working with Banach spaces $\B=C_\infty(\R^d)$ and $\D=C_\infty^2(\R^d)$. Combining these results with the theory of HJB equation leads to our main result.
\begin{theorem}
\label{thmeanfieldasequilibrium}
(i)Suppose $\{A[t,\mu,u]: t\geq 0, \mu\in \M_\lambda, \lambda\in (\lambda_1,\lambda_2),u\in\U\}$ is a family of linear operators $\D\to C_\infty(\R^d)$  of the form
\begin{equation}
A[t,\mu,u]f(z)=(h(t,z,\mu,u), \nabla f(z)) + L[t,\mu] f(z)
\end{equation}
where $h\in C([0,T]\times \U, C^{2;2,2\times 2}_{\infty}(\R^{d}\times\M_{(\lambda_1,\lambda_2)}))$ and operators $L[t,\mu]$ are of L\'evy-Khintchin form i.e.
\begin{equation}
\begin{split}
&L[t,\mu]f(z)=\frac{1}{2}(G(t,z,\mu)\nabla,\nabla)f(z)+ (b(t,z,\mu),\nabla f(z))\\
 &+\int_{\R^d} (f(z+y)-f(z)-(\nabla f (z), y){\bf 1}_{B_1}(y))\nu (t,z,\mu,dy)
\end{split}
\end{equation}
and satisfy the assumptions in Theorem \ref{thLLNgenratedevint} and estimate \eqref{EBDD} for a $p\in(0,2]$.

(ii) Let time-dependent Hamiltonian $H_t$ be given by
\begin{equation}\label{UP}
H_t(x,p, \mu):=\max_{u\in\U} \{h(t,x,\mu,u)p+J(t,x,\mu,u)\}
\end{equation}
where functions $J\in C([0,T]\times \U, C^{2;2,2\times 2}_{\infty}(\R^{d}\times\M_{(\lambda_1,\lambda_2)}))$. Assume the argmax in \eqref{UP} is unique and it is continuous in $t$ and Lipschitz continuous in $(x,p,\mu)$, uniformly with respect to $t,x,\mu$ ($\mu$ in the topology of $\D^*$) and bounded $p$ (see e.g. examples at the end of Subsection \ref{subescdiscrete}  for the natural settings where this condition holds).

(iii) Assume $\delta H_t/\delta \mu (\cdot)$ exists and is continuous and uniformly bounded.

(iv) Assume for any $\{\mu.\} \in C_\mu([0,T]\times\M_{\lambda})$ for each $\lambda\in(\lambda_1,\lambda_2)$, the operator curve $L[t,\mu_t]:\D\rightarrow C_{\infty}(\mathbf{R}^d)$ generates a strongly continuous backward propagator of bounded linear operators $U^{t,s}[\{\mu.\}]$ in the Banach space $C_{\infty}(\mathbf{R}^d)$ with the common invariant domain $\D$, such that,
for any $[\{\mu.\}]$, and $t\leq s$
\begin{equation*}
		\max \{||U^{t,s}[\{\mu.\}]||_{\D},||U^{t,s}[\{\mu.\}]||_{C_\infty}, ||U^{t,s}[\{\mu.\}]||_{C^1_\infty}\}\leq \textbf{K}
	\end{equation*}
for some constant $\textbf{K}> 0$, and with their dual propagators $\tilde{U}^{s,t}[\{\mu.\}]$ preserving the set $\M_{\lambda}$.

(v) Assume the operators $U^{t,s}[\{\mu.\}]: C_{\infty}(\mathbf{R}^d)\mapsto C^1_{\infty}(\mathbf{R}^d)$, $0 \leq t < s$, and
\begin{equation*}
\|{U}^{t,s}[\{\mu.\}]\phi\|_{C_{\infty}^1(\mathbf{R}^d)} \leq w(s-t)\|\phi\|_{C_{\infty}(\mathbf{R}^d)},
\end{equation*}
for $t<s<T$ and for all $\phi\in C_{\infty}(\mathbf{R}^d)$, with an integrable positive function $w$ on $[0,T]$.

(vi) Assume terminal function $V^T\in C\infty^2(\R^d)$.

Then,

(i) for any initial measure $\mu_0$ with a finite moment of $p$th order, $p\in (0,2]$, there exists a solution $\{\mu.\}$ to equation  \eqref{eqkineqmeanfieldnonM} ;

(ii) if the initial conditions of a $N$ players game
\[
\mu_0^N=\frac1N (\de_{X_{1,0}^N}+\cdots +\de_{X_{N,0}^N})
\]
converge in $\D^*=(C^2_{\infty}(\R^d))^*$, as $N\to \infty$, to a probability law $\mu_0 \in \PC(\R^d)$ so that
\begin{equation*}
\|\mu_0^N -\mu_0\|_{\D^*}\le \frac{\kappa_1}{N}
\end{equation*}
with a constant $\kappa_1>0$ and $X_{1,0}^N\in\R^d$ converge, as $N\to \infty$, to a point $X_{1,0}\in\R^d$ such that
\begin{equation*}
||  X^N_{1,0}-X_{1,0}||\leq \frac{\kappa_2}{N}
\end{equation*}
with a constant $\kappa_2>0$, the strategy profile $u=\Gamma(t,x,\{\mu.\})$, defined via HJB (\ref{HJB 2}) and (\ref{HJB 2a}) with $\{\hat{\mu}.\}=\{\mu.\}$, is a perfect $\epsilon$-equilibrium in a $N$ players game, with
\[
\epsilon=\frac{C(T)}{N}( ||J||_{C([0,T]\times \U, C^{2;2,2\times 2}_{\infty}(\R^{d}\times\M_{\lambda_1,\lambda_2}))}+||V^T||_{C_\infty^{2}(\R^{d})}+1 ).
\]
\end{theorem}

\proof
(i) This follows from Theorem \ref{thfeedbackHJB} and Theorem \ref{thglobalexistgenkin}.

(ii) Let the 1st player choose a different strategy $\tilde{\gamma}$, other than $\Gamma$. Denote the state dynamics of the 1st player by $X_{1,t}^N$ or $\tilde{X}_{1,t}^N$, if defined via equation (\ref{eqkineqmeanfield}) with $\gamma =\Gamma$ or $\gamma = \tilde{\gamma}$, respectively. Then,
\begin{equation*}
\begin{split}
&|V^1(0,X_{1,0}^N)[\Gamma]-V^1(0,\tilde{X}_{1,0}^N)[\Gamma_{-1}, \tilde{\gamma}]|\\
\leq & \left|\E \int_0^T  J_1(s,X_{1,s}^N,\mu_s[\Gamma], \Gamma)ds-\E\int_0^T  J_1(s,\tilde{X}_{1,s}^N,\mu_s[\Gamma_{-1},\tilde{\gamma}], \tilde{\gamma})ds\right|\\
&+ \left|\E V^T(X_{1,T}^N)-\E V^T(\tilde{X}_{1,T}^N) \right|.
\end{split}
\end{equation*}
By Theorem \ref{thLLNgenratedevinttag},
\begin{equation*}
\begin{split}
&|V^1(0,X_{1,0}^N)[\Gamma]-V^1(0,\tilde{X}_{1,0}^N)[\Gamma_{-1}, \tilde{\gamma}]|\\
\leq & \frac{C(T)}{N}\Big (||J||_{C([0,T]\times \U, C^{2;2,2\times 2}_{\infty}(\R^{d},\M_{(\lambda_1,\lambda_2)}))}+||V^T||_{C_\infty^{2}(\R^{d})}+1 \Big).
\end{split}
\end{equation*}
It is clear these estimates hold if we start the game at any time $t\in[0,T]$.
\qed

\section{Bibliographical notes}
\label{secconclud}

Below is a very short sketch (far from being complete) of relevant literature on LLN for interacting particles, kinetic equations, propagation of chaos, mean field and nonlinear Markov control.

The work on deterministic LLN limits for interacting particles and its representation in terms of nonlinear kinetic equations goes back to \cite{Leon}, \cite{Bogo}. The main approach of physicists was via the so-called Bogoliubov or BBGKY hierarchies. Similar deduction of kinetic equations
for very general system of mean-field and $k$th order interacting systems can be found in
\cite{BeM}, \cite{MaTa} and \cite{BeK3}. Quantum analogs of these equations were developed in \cite{Be88}.
The related concept of the propagation of chaos was developed by many authors, see e.g.  \cite{Kac}, \cite{Sz91}.
Especially much work was conducted for interacting diffusions and the related McKean-Vlasov nonlinear diffusion processes, see e.g. \cite{McK1} and various recent developments in \cite{YIB1},\cite{YIB2}, \cite{GMN2006},
 \cite{BogRS} and references therein.

The method of tagged particle is well established in the theory of interaction, see e.g. \cite{Fer}, \cite{Gri},
\cite{Olla}, \cite{Osa}, \cite{Pia} and references therein. However, our treatment of this method in conjunction
 with the theory of infinite-dimensional semigroups in a control setting seems to be new.

 The sensitivity analysis for kinetic equations started to be developed recently, see \cite{Ko06a}, \cite{Bai}, \cite{ManNorrisBai}, \cite{Ko10}, \cite{Ko11}, leading both to the rates of convergence in numeric algorithms
 and in theoretical result on the convergence in the dynamic LLN and the propagation of chaos. The crucial
 (especially in the context of the present paper) $1/N$-estimates appear in general setting of interacting particles in \cite{Ko07} for diffusion and stable-like processes with bounded coefficients and then was extended to a wide class of interaction with unbounded rates in \cite{Ko10}.
 Convergence rates are actively studied now, see e.g. \cite{CepFou}, \cite{Ko08a} and references therein. The algorithmical aspects of particle approximations are widely studied from various points of view, see e.g. \cite{Cris}, \cite{KSW}, \cite{DelM}, \cite{Jou} and references therein.

For the introduction to nonlinear Markov evolutions from the physical point of view we refer to papers \cite{F2008} or \cite{Z2002}. Related aspects of nonlinearity in interacting particles are discussed in \cite{Ma04} and \cite{OXZ}.

For systematic discussions of the crucial question touched upon in Section \ref{examples} on the rigorous link between generators and processes, we refer to books \cite{Ja} and \cite{VK2} and papers \cite{JS01}, \cite{Ko08a}, and references therein. Specifically for stable like processes see \cite{BC}, \cite{Barlow}, \cite{Uemura}, \cite{Ko00}
and references therein. For related questions see \cite{BogBC}, \cite{BBC}.

The seminal papers on mean-field games were cited above many times. Some further discussions can be found in video lectures \cite{L}. As a nice recent contribution, let us mention \cite{B2011}, which is devoted to a simple model with an underlying process being a standard Brownian motion, where explicit solutions and various limiting regimes including long time asymptotic are given. Discrete Mean Field Games were studied in \cite{CMS2010}, numerical methods
in \cite{AC2010}, \cite{LST2010}.   Let us stress finally  that the concept
 of the mean field games lies in the general domain of mean field and nonlinear Markov control, see
 e.g. \cite{AnDj}, \cite{Buck}, \cite{SuVv}, Section 11.2 of book \cite{Ko10} and references therein for its quick mathematical development, with more application oriented work being  \cite{LeBou}, \cite{GaGa}, \cite{Borde},
\cite{BeLe}, \cite{BeWe} (with applications from robot swamps to transportation theory and networks).

\end{document}